\documentclass[11pt,oneside]{amsart}
\usepackage{graphics}
\usepackage[dvips]{graphicx}
\usepackage{cancel}
\usepackage{amsmath,amsfonts}
\usepackage[update,prepend]{epstopdf}
\usepackage{times}
\usepackage{float}
\usepackage[ansinew]{inputenc}
\usepackage{amssymb}
\usepackage{amsmath}
\usepackage{subfigure}
\usepackage{amsfonts}
\usepackage{latexsym}
\usepackage{epsfig}
\usepackage{geometry}
\usepackage[update,prepend]{epstopdf}
\usepackage[colorlinks = true,
linkcolor = black, citecolor=black]{hyperref}
\author{ M. M. Cavalcanti, L. G. Delatorre, V. H. Gonzalez Martinez, D. C. Soares and J. P. Zanchetta}
\address{ Department of Mathematics, State University of
	Maring\'a, 87020-900, Maring\'a, PR, Brazil.}
\email{mmcavalcanti@uem.br}
\thanks{Research of Marcelo M. Cavalcanti partially supported by the CNPq Grant 300631/2003-0.}

\address{Department of Mathematics, Federal University of
	Pampa - Campus Itaqui, 97650-000, Itaqui , RS, Brazil.}
\email{leonelgdelatorre@gmail.com}

\address{Department of Mathematics, State University of
	Maring\'a, 87020-900, Maring\'a, PR, Brazil.}
\email{victor.hugo.gonzalez.martinez@gmail.com, (+55) 44 - 99728-3950}

\address{Department of Mathematics, Federal University of
	Pampa - Campus Itaqui, 97650-000, Itaqui , RS, Brazil.}
\email{daiane.mtm@gmail.com}

\address{Department of Mathematics, State University of
	Maring\'a, 87020-900, Maring\'a, PR, Brazil.}
\email{	janainazanchetta@yahoo.com.br}

\title[Klein-Gordon system with locally distributed Kelvin-Voigt damping]{Uniform stabilization for the Klein-Gordon system in an inhomogeneous medium with locally distributed Kelvin-Voigt damping}
\begin{document}

\renewcommand{\theequation}{\thesection.\arabic{equation}}
\newtheorem{theorem}{Theorem}
\newtheorem{proposition}{Proposition}[section]
\newtheorem{lemma}{Lemma}[section]
\newtheorem{corollary}{Corollary}[section]
\newtheorem{definition}{Definition}[section]
\newtheorem{remark}{Remark}
\newtheorem{assumption}{Assumption} [section]
\newtheorem{example}{Example}[section]
\newtheorem*{acknowledgement}{Acknowledgements}
\newtheorem*{notation}{Notation}

\newenvironment{dem}{\smallskip \noindent{\bf Proof}: }
{\hfill \rule{0.25cm}{0.25cm}}

\def\supp{\operatorname{{supp}}}

\begin{abstract}
 We consider the Klein-Gordon system  posed in an inhomogeneous medium $\Omega$ with smooth boundary $\partial \Omega$ subject to a local viscoelastic damping distributed around a neighborhood $\omega$ of the boundary according to the Geometric Control Condition. We show that the energy of the system goes uniformly and exponentially to zero for all initial data of finite energy taken in bounded sets of finite energy phase-space. For this purpose, refined microlocal analysis arguments are considered by exploiting ideas due to Burq and G\'erard \cite{Burq-Gerard}. By using sharp $L^p-L^q$-Carleman estimates we prove a unique continuation property for coupled systems.
\end{abstract}

\maketitle

\bigskip


\medskip



%
\tableofcontents

\section{Introduction}

\setcounter{equation}{0}

\subsection{Description of the Problem.}

This article addresses the exponential stability of a semilinear wave equation system posed in an inhomogeneous medium and subject to a Kelvin-Voigt damping locally distributed
\begin{equation}
\begin{cases}\label{1.1}
\rho(x) u_{tt} - \operatorname{div}(K(x) \nabla u) + v^2u- \operatorname{div}(a(x) \nabla u_t)=0  &~\hbox{ in }~ \Omega \times (0,T), \\
\rho(x) v_{tt} - \operatorname{div}(K(x) \nabla v) + u^2v - \operatorname{div}(b(x) \nabla v_t)=0  &~\hbox{ in }~ \Omega\times (0,T), \\
u = v =  0 &~\hbox{ on }~\partial \Omega \times (0,T),\\
u(x,0)=u_0(x), ~ u_t(x,0)=u_1(x) &~\hbox{ in }~ \Omega,\\
v(x,0)=v_0(x), ~ v_t(x,0)=v_1(x) &~\hbox{ in }~ \Omega,
\end{cases}
\end{equation}
where $\Omega \subset \mathbb{R}^{d}$, $d \leq 3$, is a bounded domain with a sufficiently smooth boundary $\Gamma=\partial \Omega$, $\rho:\Omega \rightarrow \mathbb{R}_+$, $k_{ij}:\Omega \rightarrow \mathbb{R}$, $1\leq i,j\leq d$ are $C^\infty(\Omega)$ functions such that for all $x\in \Omega$ and $\xi \in \mathbb{R}^d$,
\begin{eqnarray}\label{1.2}
\alpha_0\leq \rho(x)\leq \beta_0,\quad k_{ij}(x) = k_{ji}(x),\quad \alpha|\xi|^2 \leq \xi^{\top}\cdot K(x) \cdot \xi \leq \beta |\xi|^2,
\end{eqnarray}
where $\alpha_0,\beta_0,\alpha,\beta$ are positive constants and $K(x)=(k_{ij})_{i,j}$ is a symmetric positive-definite matrix. We denote by $\omega$, with smooth boundary $\partial \omega$, the intersection of $\Omega$ with a neighborhood of $\partial \Omega$ in $\mathbb{R}^d$.

\begin{assumption}\label{assumption1.1}
The non-negative functions $a(\cdot)$ and $b(\cdot)$, responsible for the localized
dissipative effect, satisfies the following conditions:
\begin{equation}\label{1.3}
a(\cdot), ~ b(\cdot) \in L^{\infty}(\Omega) \cap C^{0}(\overline{\omega}) \hbox{ with } a(x)\geq a_0> 0 \hbox{ in }  \omega\subset \Omega \hbox{ and } b(x)\geq b_0> 0 \hbox{ in }  \omega\subset \Omega.
\end{equation}
\end{assumption}

Let us assume that the following assumptions are also made:
	\begin{assumption}\label{assumption1.2}
		$\omega$ geometrically controls $\Omega$, i.e there exists $T_0 >0$, such that every geodesic of the metric $G(x)$,
		where $G(x)=\left(\frac{K(x)}{\rho(x)}\right)^{-1}$ travelling with speed $1$ and issued at $t = 0$, intercepts $\omega$ in a time $t <  T_0$.
	\end{assumption}

\begin{assumption}\label{assumption1.3}
	For every $T > 0$, the only solution $u$, $v \in C(]0,T[; L^2(\Omega)) \cap C(]0,T[,H^{-1}(\Omega))$ to the system
	\begin{equation}\label{1.4}
	\begin{cases}
	\rho(x) u_{tt} - \operatorname{div}(K(x) \nabla u) + V_1(x,t)u=V_3(x,t)v  &~\hbox{ in }~ \Omega \times (0,T), \\
	\rho(x) v_{tt} - \operatorname{div}[K(x) \nabla v] + V_2(x,t)v=V_3(x,t)u  &~\hbox{ in }~ \Omega\times (0,T), \\
	u=v=0 &~\hbox{ on }~ \omega,\\
	\end{cases}
	\end{equation}
	where  $V_1(x,t)$, $V_2(x,t)$ and $V_3(x,t)$ are elements of $L^\infty( ]0,T[,L^{\frac{d+1}{2}}(\Omega))$, is the trivial
	one $u=v= 0$.
\end{assumption}

Assumption \ref{assumption1.2} is the so called Geometric Control Condition (G.C.C.). It is well-known that it is necessary and sufficient for stabilization and control of the linear wave equation, see \cite{Bardos}, \cite{Burq-Gerard-CR}, \cite{Cavalcanti3}, \cite{Cavalcanti2}, \cite{Dehman2}, \cite{RT} and the references therein. In this direction, it is worth mentioning the work due to Betel\'u, Gulliver and Littman \cite{Littman} whom have discussed the question of {\bf closed geodesics} in the interior of a region $\Omega$. Such closed geodesics make control impossible since they are bicharacteristics that never reach the controlled boundary. In this paper, they show that, in the two-dimensional case, the lack of closed geodesics in the interior is also sufficient for control. For this reason and since in the present paper we do not have any control of the geodesics because of the inhomogeneous medium we are forced to consider $\omega$ a neighbourhood of the whole boundary $\partial \Omega$ and to give conditions on the metric $G=(K/\rho)^{-1}$ in order every geodesic of the metric $G$ enters the set $\omega$ in a time $t<T_0$.

\begin{remark}\label{remark1}
	It is important to observe that Assumption \ref{assumption1.2} is not obviously fulfilled for every matrix $G=(K/\rho)^{-1}$. For instance, the equator in the unit sphere $S^2$ is a periodic geodesic and we can consider $\Omega \subset S^2$ that contains the equator as a domain in $\mathbb{R}^2$ endowed with a Riemannian metric $G$. We shall give in the appendix of this manuscript examples where this situation occurs. An easy one happens when $G(x)=I_d$. In this case the geodesics are straight lines and necessarily they will meet the region $\omega$. We emphasize that the example presented in the appendix, that satisfies Assumption \ref{assumption1.2}, is valid not  only for Kelvin-Voigt damping but also for frictional damping.
\end{remark}

\begin{remark}\label{remark2}
In the next section we apply the Carleman estimates due to Dos Santos Ferreira \cite{David} and Koch and Tataru \cite{Tataru} to show certain situations where the unique continuation property given in the Assumption \ref{assumption1.3} holds. Although the result has been proved for a single equation, a new proof is required with respect to coupled systems of wave equations, which, to our surprise, has not been done so far (at least at the level of the above-mentioned articles). The present work seems to be the pioneer in proving a unique continuation property for coupled systems by using sharp $L^p-L^q$-Carleman estimates. In this direction the authors would like to thank Daniel Tataru for fruitful discussions regarding this issue. However, it is worth mentioning the most recent paper regarding this issue due to Tebou \cite{Tebou1}. The author considers a system of two coupled nonconservative wave equations. For this system, he proves several observability estimates. Those observability estimates
are sharp in the sense that they lead by duality to the controllability (exact or approximate) of the coupled system with a single control acting through one of the equations only while keeping the same controllability time as for a single equation. Unfortunately the aforementioned result seems not be applicable in the present paper.
\end{remark}

\subsection{Main Goal, Methodology and Previous Results.}

The main objective of the present manuscript is to prove the existence
and uniqueness for weak solutions to problem \eqref{1.1} and, in addition, that
those solutions decay exponentially and uniformly to zero, that
is, setting
\begin{align}\label{1.5}
E_{u,v}(t)= {} & \frac{1}{2}\int_\Omega \rho(x)| u_t(x,t)|^2+\rho(x)|v_t(x,t)|^2 +  \nabla u(x,t)^{\top} \cdot K(x) \cdot\nabla u(x,t) \,dx \nonumber\\
& {} +{} \frac{1}{2}\int_{\Omega}\nabla v(x,t)^{\top} \cdot K(x) \cdot\nabla v(x,t)+ (uv)^2(x,t) \,dx,
\end{align}
there exist positive constants $ C, \gamma,$
such that
\begin{eqnarray}\label{1.6}
E_u(t)\leq C e^{-\gamma\,t}E_u(0), \hbox{ for all } t\geq T_0,
\end{eqnarray}
for all weak  solutions to problem \eqref{1.1} provided that the initial data $\{u_0, v_0, u_1, v_1 \}$ are taken in bounded sets of $H_0^1(\Omega)\times H_0^1(\Omega)\times L^2(\Omega) \times  L^2(\Omega)$. This result is a {\em local stabilization result}. Indeed, the constants $C$ and $\gamma$ are uniform on every ball in $H_0^1(\Omega)\times H_0^1(\Omega)\times L^2(\Omega) \times  L^2(\Omega)$ with radius $R>0$ of the energy space but the result does not guarantee that the decay rate is global one, i.e. whether \eqref{1.6} holds with constants $C, \gamma$ which are independent of the initial data.

Inspired in Dehman, G\'erard, Lebeau, \cite{Dehman} or Dehman, G. Lebeau and Zuazua \cite{Dehman2}  we give a direct proof of the inverse inequality to problem \eqref{1.1}, namely, we prove that given $T>0$ there exists a constant positive $C=C(T)$ such that
\begin{eqnarray}\label{1.7}
E_{u,v}(0) \leq C \int_0^T\int_\Omega a(x) |\nabla u_t(x,t)|^2+b(x) |\nabla v_t(x,t)|^2\,dxdt,
\end{eqnarray}
provided the initial data are taken in bounded sets of $H_0^1(\Omega) \times L^2(\Omega)$.

To prove \eqref{1.7} and therefore the stability resut, we argue by contradiction and we find a sequence of $\{w^n,z^n\}$ of weak solutions to problem \eqref{1.1} such that $E_{w^n,z^n}(0)=1$. In order to obtain a contradiction we need to prove that $E_{w^n,z^n}(0)\rightarrow 0$ as $n \rightarrow +\infty$. We shall prove, by exploiting the properties of $K(x)$, $a(x)$, $b(x)$ and a unique continuation principle, that
\begin{eqnarray}\label{1.8}
\int_0^T \int_{\omega}   |w_t^n|^2  \, dxdt \rightarrow 0 \hbox{ and } \int_0^T \int_{\omega}   |z_t^n|^2  \, dxdt \rightarrow 0
\end{eqnarray}
when $n$ goes to infinity.

Our wish is to propagate the convergence \eqref{1.8} from $\omega \times (0,T)$ to the whole set $\Omega \times (0,T)$. In order to do this, we are able to  successfully use the microlocal analysis. Indeed we consider the microlocal defect measure $\mu$, in short m.d.m., firstly introduced by G\'erard \cite{Gerard} associated to the solution of the linear wave equation. We note that once one has
\begin{equation}\label{1.9}
\begin{aligned}
\rho(x)\partial_t^2 w_t^n - \operatorname{div}(K(x)\nabla w_t^n) \rightarrow 0~ \hbox{ in }~H^{-1}_{loc}(\Omega \times (0,T)),  \\
\rho(x)\partial_t^2 z_t^n - \operatorname{div}(K(x)\nabla z_t^n) \rightarrow 0~ \hbox{ in }~H^{-1}_{loc}(\Omega \times (0,T)),
\end{aligned}
\end{equation}
so, using properties associated to $\mu$  we are able to prove that $\mu$ propagates along the bicharacteristic flow of the wave operator $\rho(x)\partial_{tt}^{2}-\operatorname{div}(K(x) \nabla (\cdot))$ proving the desired convergence. It is important to observe that, in the present article, in order to avoid certain technicalities that may arise when considering the propagation up to the boundary, we consider $\omega$, the region where the damping is effective, as a neighborhood of the boundary and satisfying the
GCC. Under these conditions, we can use the propagation results found in \cite{Burq-Gerard}. This will be clarified in section 3 and in the appendix of the present manuscript.

		The stability of the single wave equation subject to a Kelvin-Voigt damping $$\rho(x) u_{tt} - \operatorname{div}(K(x) \nabla u) - \operatorname{div}(a(x) \nabla u_t)=0,~(0,L) \times (0,\infty),$$ becomes susceptible to the continuity of the materials, in the sense that, it has already been proved for the one dimensional case by Liu and Liu in \cite{Liu0} that if the damping coefficient $a$ is discontinuous across the interface of the materials, the energy does not decay uniformly. Obviously, the same occurs for coupled one-dimensional wave systems subject two dampings of Kelvin-Voigt type. In a higher dimensional setting, even with a more regular damping coefficient and considering a smooth initial data, as explained by Liu and Rao \cite{Liu}, there is a loss of regularity of the solutions, which makes it more difficult to use the usual multiplier methods. Hence, the multiplier method needs to be combined with other techniques in order to overcome these difficulties. The method then leads the authors to impose several technical conditions on the damping coefficient. In the work of Tebou in \cite{Tebou}, such conditions on the damping coefficient as well as the conditions on the feedback control region are relaxed.  However, the author still requires an inequality  constraint on the gradient of the damping coefficient, which will not be required in our current study.  

There are two main difficulties regarding problem \eqref{1.1}. Since we are dealing with localized  Kelvin-Voigt type dampings, in a neighborhood of the boundary, since the operators defining the dampings are unbounded. Moreover, it is worth mentioning that the presence of the coefficients in the wave operators, as considered in the present paper, makes the analysis much more refined in terms of the rays of the geometrical optics. We also have that the dampings are effective only on the set where $a(x) >0$ and $b(x)>0$, we have an interaction between an elastic material (portion of $\Omega$ where $a =b \equiv 0$) and a viscoelastic material (portion of $\Omega$ when $a>0$ and $b>0$).  By virtue of the nature of the damping there is a loss of regularity of the solutions, which makes it more difficult to use the standard multiplier methods. Hence, the multiplier method needs to be combined with other techniques in order to overcome these difficulties. The great advantage of using refined microlocal analysis arguments is that we can use the geodesic flow in our favor, and for this we need to consider $ n \geq 2 $, in order to direct the geodesics so that they touch the boundary of $ \Omega $.  This is only possible for certain special cases involving the metric $G=(K/\rho)^{-1}$.  This methodology allows us not only to prove the exponential stability of linear problems but also semi-linear problems, which remained an open question until now. The main ingredients in the proof are: (i) a new unique continuation principle for systems, (ii) the propagation of the microlocal defect measure by the geodesic flow as previously mentioned. This is the main contribution of the present article.

The model proposed in this article is inspired by an equation introduced by Segal in \cite{Segal}, given by
\begin{equation}\label{Segal}
\begin{cases}
u_{tt} - \Delta u + m_{1}^2u + gv^2u = 0  &~\hbox{ in }~ \Omega \times (0,T), \\
v_{tt} - \Delta v + m_{2}^2v   + hu^2v = 0  &~\hbox{ in }~ \Omega\times (0,T),
\end{cases}
\end{equation}
which describes the interaction of scalar fields $u, v$ of mass $m_1, m_2$, respectively, with interaction constants $g$ and $h$. This system defines the motion of charged mesons in an electromagnetic field. As the interest of this paper is to make the mathematical analysis, there is no loss of generality if we consider only the case in which $m_1=m_2=0$ and $g=h=1$.

The problem \eqref{1.1}, when $m_1=m_2=0$, and $g=h=1$ with Dirichlet boundary condition on $\partial \Omega$, was studied by Medeiros and Menzala \cite{Medeiros}. In this paper, the authors
proved the existence and uniqueness of global weak solutions, using Galerkin method, provided that $d \leq 3$.  Further generalizations are also
given in \cite{Medeiros1} and \cite{Medeiros2} by using Galerkin methods, where the authors consider nonlinearities of the form $|v|^{\rho+2}|u|^\rho u$, $|u|^{\rho+2}|v|^{\rho}v$ and some hypotheses about the coefficient $\rho$, which are related to the dimension $d$ of the space. Related to global existence and uniqueness of solutions we also have the article \cite{Andrade}, where the authors prove the existence of solutions to a Klein-Gordon system with memory and nonlinearities similar to those considered in \cite{Medeiros1}  and \cite{Medeiros2}. Another generalization for this system can be founded in \cite{Cousin}, where the authors consider a $k \times k$ system os Klein-Gordon equations with acoustic boundary conditions.

In \cite{Ferreira1} Ferreira deduce decay results for the solutions of the system of nonlinear Klein-Gordon equations in $\mathbb{R}^3$, more precisely, if the inital datas are taken in $C_{0}^{\infty}(\mathbb{R}^{3})$ then, every solution of the system decays uniformly in $x$ at the rate $\mathcal{O}(t^{-\frac{3}{2}})$ as $|t|\rightarrow \infty$, and is asymptotic to a free solution at $t=\infty$ or $t=-\infty$. To get these results, they use techniques developed by Morawetz and Strauss. 

 We also would like to quote the papers \cite{Ferreira4} and \cite{Ferreira2}. In these papers, the exponential stability for the system \eqref{Segal} is established, under the assumption of two frictional damping terms $a(x) u_t$ and $b(x) v_t$ instead of Kelvin-Voigt dampings, as considered in the present article, where the proof also works for any combination of dampings of frictional or Kelvin-Voigt's type. In addition, in \cite{Ferreira4} and \cite{Ferreira2} the problem is considered in a homogeneous medium which is easier to be analysed since the bicharacteristics are straight lines.  We also note that in \cite{Ferreira4} and \cite{Ferreira2} the author mentions the pioneering work of Ruiz \cite{Ruiz}, which ensures the unique continuation property for a single equation with potential $V(t,x) \in L^\infty(]0,T[, L^n(\Omega))$ and the Euclidean metric $G(x)= I_{d}$, and assumes, without proof,  that the result is valid for coupled systems. In the present paper we solve this question once and for all. Finally, in  \cite{Ferreira4} and \cite{Ferreira2}, due to the type of dissipation used, it is possible to obtain the existence of regular solutions and consequently the use of radial multipliers in order to prove the exponential stability , which can not be definitely done in the context of this article.
In the paper \cite{Cavalcanti4}, the authors consider a generalized Klein-Gordon system consisting of two coupled semilinear evolution equations of mixed degenerate hyperbolic-parabolic type. Boundary conditions are assumed to be dissipative as well as global existence and asymptotic decay rates are proved.

In \cite{Ferreira3}, the authors have been study time decay properties of solutions of the problem \eqref{Segal}. They show that, in the $L^\infty(\mathbb{R}^{3})$ norm, the solutions decay like $\mathcal{O}(t^{-\frac{3}{2}})$ as $t \rightarrow + \infty$ provided the initial data are sufficiently small. After this, the author prove that finite energy solutions of such a system decay in local energy norm as $t \rightarrow + \infty$. To our knowledge, this was the first work to address the existence and uniqueness of solutions to this system through the semigroups theory.

Our paper is organized as follows. In Section 2 we verify that the unique continuation property holds for the coupled system. In Section 3 we give some notation and we establish the well-posedness to problem \eqref{1.1}. In Section 4 and in the Appendix we give the proof of the stabilization which consists our main result.

\section{Unique Continuation Property}
\setcounter{equation}{0}
Let us now recall some results on the unique continuation property for differential operators. For this, we will use the notations of Dos Santos Ferreira  \cite{David} and Koch-Tataru \cite{Tataru}.

Let $P(x,D)$ be a second order differential opertor of real principal type with $C^\infty$ coefficients and $u$ be the solution of the differential equation
\begin{equation}\label{2.1}
P(x,D)u+V(x)u=0, ~ x \in \Omega \subset \mathbb{R}^{d}
\end{equation}
where $V \in L^{\frac{d}{2}}_{loc}$ and let $S$ be a $C^\infty$ hypersurface in $\mathbb{R}^{d}$ locally defined by
\begin{eqnarray}\label{2.2}
S_{+}&=&\{x \in \Omega: \psi(x) > \psi(x_0)  \} \nonumber \\
S_{-}&=&\{x \in \Omega: \psi(x) < \psi(x_0)  \}.
\end{eqnarray}
We say that a solution of \eqref{2.1} has unique continuation across the hypersurface $S$ if when $u$ vanishes on the side $S_{-}$ then it vanishes on a whole neighbourhood of $x_0$.

\begin{definition}\label{definition2.1}
	The hypersurface $S=\{x \in \Omega: \phi(x)=\phi(x_0)   \}$ ie said to be (strictly) pseudo-convex at $x_0 \in \Omega$ with respect to the real principal type differential operator $P$ of oreder $2$ whose principal symbol is $p$ whenerver\begin{equation*}
	p(x_0,\xi)=H_p\phi(x_0,\xi)=0 \Rightarrow H_{p}^{2}(x_0,\xi)<0  \hbox{ for all } \xi \in \mathbb{R}^{d}.
	\end{equation*}
\end{definition}

Using $(L^{p}-L^{q})$ Carleman estimates, the following result is proved by Dos Santos Ferreira in \cite{David}
\begin{proposition}\label{proposition2.1}
	Let $P(x,D)$ be a differential operator of order $2$, defined on an open set $\Omega \subset \mathbb{R}^{d}$, $n>2$. Suppose that there exists $M \geq 1$, a neighbouhood $\Omega_0$ of $x_0$ and a function $\phi \in C^\infty$ verifying $\{ x \in \Omega_0: x \neq x_0, \phi(x)\leq \phi(x_0)  \} \subset S_{-}$ such that the Carleman estimate
	\begin{equation}\label{2.3}
	\|e^{-\tau \phi}v\|_{L^{\frac{2d}{d-2}}}\leq C \|e^{-\tau}\phi P(x,D)v\|_{L^{\frac{2d}{d+2}}}
	\end{equation}
	holds for all $u \in C_{0}^{\infty}(\Omega_0)$ and $\tau \geq M$. Then if $u \in H^1$ is the solution of the equation \eqref{2.1} on $\Omega$ where $V \in L^{\frac{d}{2}}_{loc}$ and $u$ vanishes on $S_{-}$, then exist a neighbourhood of $x_0$ on which u vanishes.
\end{proposition}

Inspired in this Proposition, we obtain the following result:

\begin{proposition}\label{proposition2.2}
	Let $P_1(x,D)$ and $P_2(x,D)$ be two differential operator of order $2$, defined on an open set $\Omega \subset \mathbb{R}^{d}$, $n>$. Suppose that there exists $M \geq 1$, a neighbouhood $\Omega_0$ of $x_0$ and a function $\phi \in C^\infty$ verifying $\{ x \in \Omega_0: x \neq x_0, \phi(x)\leq \phi(x_0)  \} \subset S_{-}$ such that the following Carleman estimates
	\begin{eqnarray}\label{2.4}
	\|e^{-\tau \phi}u\|_{L^{\frac{2d}{d-2}}}\leq C \|e^{-\tau\phi} P_1(x,D)u\|_{L^{\frac{2d}{d+2}}}
	\end{eqnarray}
	and
	\begin{eqnarray}\label{2.5}
	\|e^{-\tau \phi}v\|_{L^{\frac{2d}{d-2}}}\leq C \|e^{-\tau\phi} P_2(x,D)v\|_{L^{\frac{2d}{d+2}}}	
	\end{eqnarray}
	holds for all $u$, $v \in C_{0}^{\infty}(\Omega_0)$ and $\tau \geq M$. Then the unique continuation property holds for the solution of the coupled system
	\begin{equation}\label{2.6}
	\begin{cases}
	P_1(x,D)u+V_1u= &V_3v,\\
	P_2(x,D)v+V_2v=V_3u,
	\end{cases}
	\end{equation}
	where $V_1$, $V_2$ and $V_3$ are elements of $L^{\frac{d}{2}}_{loc}$.
\end{proposition}
\begin{proof}
	Suppose $x_0=0$ and $\psi(0)=0$ and let $\chi \in C_{0}^{\infty}(\Omega_0)$ equal $1$ on a neighbourhood $\Omega_1$ of $0$. Applying the Carleman inequality to the functions $w_1=\chi u$, $w_2=\chi v$ and observing that $L^{\frac{d}{2}} \cdot L^{\frac{2d}{d-2}} \subset L^{\frac{2d}{d+2}}$, we obtain
	\begin{align*}
		\|e^{-\tau \phi} w_1\|_{L^{\frac{2d}{d-2}}}  \lesssim {} & \|e^{-\tau \phi} P_1(x,D)w_1\|_{L^{\frac{2d}{d+2}}}\\
		 \lesssim {} &  \|e^{-\tau \phi} P_1(x,D)(\chi u)\|_{L^{\frac{2d}{d+2}}}\\
		 \lesssim {} & \|e^{-\tau \phi}( \chi P_1(x,D)u+[P_1,\chi]u)\|_{L^{\frac{2d}{d+2}}}\\
		 \lesssim {} & \|V_3e^{-\tau \phi} \chi v -V_1e^{-\tau \phi}\chi u+e^{-\tau \phi}[P_1,\chi]u\|_{L^{\frac{2d}{d+2}}}\\
		 \lesssim {} & \|V_3\|_{L^{\frac{d}{2}}(K)}\|e^{-\tau \phi}\chi v\|_{L^{\frac{2d}{d-2}}}+\|V_1\|_{L^{\frac{d}{2}}(K)}\|e^{-\tau \phi}\chi u\|_{L^{\frac{2d}{d-2}}}\\
		&{}+\|e^{-\tau \phi}[P_1,\chi]u\|_{L^{\frac{2d}{d+2}}},
	\end{align*}
	where $K=\supp(\chi)$. That is,
	\begin{align}\label{2.7}
	\|e^{-\tau \phi} w_1\|_{L^{\frac{2d}{d-2}}}{} \lesssim {}&  \|V_3\|_{L^{\frac{d}{2}}(K)}\|e^{-\tau \phi}\chi v\|_{L^{\frac{2d}{d-2}}}+\|V_1\|_{L^{\frac{d}{2}}(K)}\|e^{-\tau \phi}\chi u\|_{L^{\frac{2d}{d-2}}} \nonumber  \\
	&{}+ \|e^{-\tau \phi}[P_1,\chi]u\|_{L^{\frac{2d}{d+2}}}.
	\end{align}
	Analogously,
	\begin{align}\label{2.8}
	\|e^{-\tau \phi} w_2\|_{L^{\frac{2d}{d-2}}} {} \lesssim {} & \|V_3\|_{L^{\frac{d}{2}}(K)}\|e^{-\tau \phi}\chi u\|_{L^{\frac{2d}{d-2}}}+\|V_2\|_{L^{\frac{d}{2}}(K)}\|e^{-\tau \phi}\chi v\|_{L^{\frac{2d}{d-2}}} \nonumber  \\
	& {} + {} \|e^{-\tau \phi}[P_2,\chi]v\|_{L^{\frac{2d}{d+2}}}.
	\end{align}
		Adding inequalities \eqref{2.7} and \eqref{2.8} and organizing the terms, we obtain
	\begin{align}\label{2.9}
	\|e^{-\tau \phi} \chi u\|_{L^{\frac{2d}{d-2}}} +\|e^{-\tau \phi} \chi v\|_{L^{\frac{2d}{d-2}}}  \lesssim {}  &  (\|V_1\|_{L^{\frac{d}{2}}(K)}+\|V_3\|_{L^{\frac{d}{2}}(K)})\|e^{-\tau \phi}\chi u\|_{L^{\frac{2d}{d-2}}}\nonumber  \\
	& {} + {} (\|V_2\|_{L^{\frac{d}{2}}(K)}+\|V_3\|_{L^{\frac{d}{2}}(K)})\|e^{-\tau \phi}\chi v\|_{L^{\frac{2d}{d-2}}}  \nonumber \\
	& {} + {} \|e^{-\tau \phi}[P_1,\chi]u\|_{L^{\frac{2d}{d+2}}}+\|e^{-\tau \phi}[P_2,\chi]v\|_{L^{\frac{2d}{d+2}}}.
    \end{align}
	Then we choose $\chi$ with sufficiently small support so that the two first terms on the right handside of \eqref{2.9} may be absorbed in the left hand side term, thus
	\begin{eqnarray}\label{2.10}
	\|e^{-\tau \phi} \chi u\|_{L^{\frac{2d}{d-2}}} +\|e^{-\tau \phi} \chi v\|_{L^{\frac{2d}{d-2}}} & \lesssim & \|e^{-\tau \phi}[P_1,\chi]u\|_{L^{\frac{2d}{d+2}}}+\|e^{-\tau \phi}[P_2,\chi]v\|_{L^{\frac{2d}{d+2}}}.
	\end{eqnarray}
	But $[P_1(x,D),\chi]u$ and $[P_2(x,D),\chi]v$ are classical differential operators of order $1$ supported in $\supp u \cap \overline{\Omega_0 \setminus \Omega_1} \subset \{ \phi>0\}$ and  $\supp v \cap \overline{\Omega_0 \setminus \Omega_1} \subset \{\phi>0\}$ where one has $\phi>c>0$, which implies that
	\begin{eqnarray*}
		\|e^{-\tau \phi} \chi u\|_{L^{\frac{2d}{d-2}}} +\|e^{-\tau \phi} \chi v\|_{L^{\frac{2d}{d-2}}} & \lesssim & e^{-\tau c} (\|[P_1(x,D),\chi]u\|_{L^{\frac{2d}{d+2}}}+\|[P_2(x,D),\chi]v\|_{L^{\frac{2d}{d+2}}})\\
		& \lesssim & e^{-\tau c} (\|u\|_{H^1}+\|v\|_{H^1}).
	\end{eqnarray*}
	Therefore,
	\begin{align*}
		\|e^{-\tau (\phi-c)} \chi u\|_{L^{\frac{2d}{d-2}}} +\|e^{-\tau (\phi-c)} \chi v\|_{L^{\frac{2d}{d-2}}}
		 \lesssim {} & \|u\|_{H^1}+\|v\|_{H^1},
	\end{align*}
	that is, $\|e^{-\tau (\phi-c)} \chi u\|_{L^{\frac{2d}{d-2}}} +\|e^{-\tau (\phi-c)} \chi v\|_{L^{\frac{2d}{d-2}}}$ is bounded, which is impossible unless $u=v=0$ when $\phi < c$. This complete the proof of the unique continuation.
\end{proof}

Similar arguments can be applied to prove unique continuation property for second order hyperbolic coupled systems with nonsmooth coefficients and potentials in $L^{\frac{d+1}{2}}$. Indded, in Koch-Tataru we can find the following result:

\begin{theorem}\label{theorem1}
	Let $P$ be a second-order hyperbolic operator with $C^2$ coefficients. Let $\phi$ be a strictly pseudo-convex function with respect to $P$. Then for compactly supported $u$ we have
	\begin{equation}\label{2.11}
	\|e^{\tau \phi}u\|_{L^{\frac{2(d+1)}{d-1}}\cap \tau^{-\frac{1}{4}}H_{\tau}^{\frac{1}{2}}} \lesssim \|e^{\tau \phi } P(x,D)u\|_{L^{\frac{2(d+1)}{d+3}}+\tau^{\frac{1}{4}}H_{\tau}^{-\frac{1}{4}}}, ~ \tau >\tau_0.
	\end{equation}
\end{theorem}

Let $X=L^{\frac{2(d+1)}{d-1}}$ and $Y=\tau^{-\frac{1}{4}}H_{\tau}^{\frac{1}{2}}$, then the inequality \eqref{2.11} can be rewrite as
\begin{equation}\label{2.12}
\|e^{\tau \phi}u\|_{X \cap Y} \lesssim \|e^{\tau \phi } P(x,D)u\|_{X^\ast+Y^\ast}, ~ \tau >\tau_0.
\end{equation}
Note that $X \cap Y \hookrightarrow X$, which implies that $X^\ast \hookrightarrow X^\ast+Y^\ast$. From the inequality \eqref{2.12} we deduce that,
\begin{equation}\label{2.13}
\|e^{\tau \phi}u\|_{X} \lesssim \|e^{\tau \phi } P(x,D)u\|_{X^\ast}, ~ \tau >\tau_0.
\end{equation}

Thus, we deduce the following result:

\begin{proposition}\label{proposition2.3}
	Let $P_1(x,D)$ and $P_2(x,D)$ be two second-order hyperbolic operators with $C^2$ coefficients. Let $\phi$ be a strictly pseudo-convex function with respect to $P_1$ and $P_2$. Then the unique continuation property holds for the solution of the coupled system
	\begin{equation}\label{2.14}
	\begin{cases}
	P_1(x,D)u+V_1u=V_3v\\
	P_2(x,D)v+V_2v=V_3u,
	\end{cases}
	\end{equation}
	where $V_1$, $V_2$ and $V_3$ are elements of $L^{\frac{d+1}{2}}$.
\end{proposition}

\section{Well-Posedness}
\setcounter{equation}{0}

We consider the weak phase space
\begin{equation*}
\mathcal{H} = H_{0}^{1}(\Omega) \times H_{0}^{1}(\Omega) \times  L^{2}(\Omega) \times  L^{2}(\Omega),
\end{equation*}
which is endowed with the inner product
\begin{equation*}
\langle (u_{1}, u_{2}, u_{3}, u_{4}), (v_{1}, v_{2}, v_{3}, v_{4}) \rangle_{\mathcal{H}} = \int_{\Omega} \nabla u_{1}^{\top}\cdot K(x) \cdot \nabla v_{1}+\nabla u_{2}^{\top}\cdot K(x) \cdot \nabla v_{2} + \rho u_{3}v_{3}+\rho u_{4}v_{4} \,dx.
\end{equation*}
Denoting $W(t)=(u,v,u_t,v_t)$ we may rewrite problem \eqref{1.1} as the following Cauchy problem in $\mathcal{H}$
\begin{equation}\label{3.1}
\left\{
\begin{aligned}
\frac{\partial W}{\partial t}(t) = {} & \mathcal{A}W(t) + \mathcal{F}(W(t)) \\
 W(0)={} & (u_{0}, v_{0}, u_{1}, v_{1}),
\end{aligned}
\right.
\end{equation}
where the linear unbounded operator $\mathcal{A}: D(\mathcal{A}) \rightarrow \mathcal{H}$ is given by
\begin{equation}\label{3.2}
\mathcal{A}=\left(
\begin{array}{cccc}
0 & 0 &I&0 \\
0 & 0 & 0&I \\
\frac{1}{\rho}\operatorname{div}(K(x) \nabla (\cdot)) & 0 & \frac{1}{\rho}\operatorname{div}(a(x) \nabla (\cdot)) &0\\
0 &\frac{1}{\rho}\operatorname{div}(K(x) \nabla (\cdot)) & 0&\frac{1}{\rho}\operatorname{div}(b(x) \nabla (\cdot))
\end{array}
\right),\end{equation}
that is,
\begin{eqnarray}\label{3.3}
\mathcal{A}(u,v,w,z) = \left( w,z,\frac{1}{\rho}\operatorname{div}(K(x) \nabla u+a(x) \nabla w), \frac{1}{\rho} \operatorname{div}(K(x)\nabla v +b(x) \nabla z)\right),
\end{eqnarray}
with domain
\begin{equation}\label{3.4}
D(\mathcal{A}) = \{ (u,v,x,z) \in \mathcal{H} : w,z \in H_{0}^{1}(\Omega), \operatorname{div}(K(x) \nabla u + a(x)\nabla w), \operatorname{div}(K(x) \nabla v + a(x)\nabla z) \in L^{2}(\Omega)\}
\end{equation}
and $\mathcal{F}: \mathcal{H} \rightarrow \mathcal{H}$ is the nonlinear operator
\begin{equation}\label{3.5}
\mathcal{F}(u,v, w, z) = \left(0, 0, -\frac{1}{\rho}v^2u, -\frac{1}{\rho}u^2v\right).
\end{equation}
Now, we are in conditions to state the well-posedness result for problem \eqref{3.1}, which ensures that problem \eqref{1.1} is globally well-posed.

\begin{theorem}[Global Well-Posedness]\label{theorem2}
	Assume that the hypotheses on $\rho$ and $K$ are fullfiled and the initial data $(u_{0}, v_0, u_{1},  v_{1}) \in \mathcal{H}$. Then problem \eqref{3.1} possesses a unique  mild solution $W \in C([0,T]; \mathcal{H})$. Moreover, if $(u_{0}, v_{0}, u_{1}, v_{1}) \in D(\mathcal{A})$, then the solution is regular.
\end{theorem}

\begin{proof}
	First of all the operator $\mathcal{A}: D(\mathcal{A}) \subset  \mathcal{H} \rightarrow \mathcal{H}$ defined by \ref{3.4} and \eqref{3.5} generates a $C_{0}$-semigroup of contractions $e^{\mathcal{A}t}$ on the energy space $\mathcal{H}$. Indeed, it is easy to see that for all $(u, v, w, z) \in D(\mathcal{A})$, we have
	\begin{eqnarray}\label{3.6}
	(\mathcal{A}U,U)=-\int_{\Omega}a(x)|\nabla w|^2+b(x)|\nabla z|^2 \, dx \leq 0,
	\end{eqnarray}
	which show that the operator $\mathcal{A}$ is dissipative. Next, for any given $(p,q,r,s) \in \mathcal{H}$, we solve the equation $(I-\mathcal{A})(u, v, w, z)=(p, q, r, s)$ which is recast on the following way
	\begin{equation}\label{3.7}
	\begin{cases}
	u-w=p,\\
	v-z=q,\\
	w-\frac{1}{\rho}\operatorname{div}(K(x)\nabla u+a(x)\nabla w)=r,\\
	z-\frac{1}{\rho}\operatorname{div}(K(x)\nabla v+b(x)\nabla z)=s,
	\end{cases}
	\end{equation}
	From \eqref{3.7} we obtain
    \begin{equation}\label{3.8}
	\left\{\begin{aligned}
	w={} & u-p,\\
	z={} &v-q.
	\end{aligned}\right.
	\end{equation}
	Substituting the equations of \eqref{3.8} in the last two equations of \eqref{3.7}, we obtain
	\begin{equation}\label{3.9}
	\left\{\begin{aligned}
	u-\frac{1}{\rho}\operatorname{div}(K(x)\nabla u+a(x)\nabla                    u)= {}&-\frac{1}{\rho}\operatorname{div}(a(x)\nabla p)+p+r,  \\
	v-\frac{1}{\rho}\operatorname{div}(K(x)\nabla v+a(x)\nabla v) ={}&-\frac{1}{\rho}\operatorname{div}(b(x) \nabla q)+q+s.
	\end{aligned}\right.
	\end{equation}
	Equivalently,
	\begin{equation}\label{3.10}
	\left\{\begin{aligned}
	\rho u-\operatorname{div}(K(x)\nabla u+a(x)\nabla u)={}&-\operatorname{div}(a(x)\nabla p)+\rho (p+ r),  \\
	\rho v-\operatorname{div}(K(x)\nabla v+b(x)\nabla v)={}&-\operatorname{div}(b(x) \nabla q)+\rho( q+ s).
	\end{aligned}\right.
	\end{equation}
	Define
	$$a: H_{0}^{1}(\Omega) \times H_{0}^{1}(\Omega) \times H_{0}^{1}(\Omega) \times H_{0}^{1}(\Omega) \longrightarrow \mathbb{R}$$
	given by
	\begin{equation}
	\begin{aligned}\label{3.11}
	a((u_1,v_1),(u_2,v_2))= {} &\int_{\Omega} \rho u_1u_2+ (\nabla u_1)^{ \top } K(x) \nabla u_2 +(\nabla u_1)^{ \top } a(x) \nabla u_2 \, dx \\
	& {} + {} \int_{\Omega} \rho v_1v_2+ (\nabla v_1)^{ \top } K(x) \nabla v_2 +(\nabla v_1)^{ \top } b(x) \nabla v_2 \, dx.
    \end{aligned}
    \end{equation}
	Then, $a(\cdot,\cdot)$ is bilinear, continuous and coercive. Thank's to the Lax-Milgram Theorem, given $$(a_0,b_0)=(-\operatorname{div}(a(x)\nabla p)+\rho(p+r), -\operatorname{div}(b(x) \nabla q)+\rho (q+  s) \in H^{-1}(\Omega)\times H^{-1}(\Omega),$$
	there exists a unique $(u,v) \in H_{0}^{1}(\Omega) \times H_{0}^{1}(\Omega)$, such that
	\begin{equation*}
	a((u,v),(x,y))=\langle a_0,x \rangle_{H^{-1}(\Omega),H_{0}^{1}(\Omega)}+\langle b_0,y \rangle_{H^{-1}(\Omega),H_{0}^{1}(\Omega)} \hbox{ for all } (x,y) \in H_{0}^{1}(\Omega)\times H_{0}^{1}(\Omega).
	\end{equation*}
	Taking $x \in C_{0}^{\infty}(\Omega)$ and $y=0$ we deduce that
	\begin{equation*}
	\rho u-\operatorname{div}(K(x)\nabla u+a(x)\nabla u)=a_0 \hbox{ in } \mathcal{D}'(\Omega)
	\end{equation*}
	and consequently,
		\begin{equation*}
	\rho u-\operatorname{div}(K(x)\nabla u+a(x)\nabla u)=a_0 \hbox{ in } H^{-1}(\Omega).
	\end{equation*}
	Proceeding in the same way, we deduce that
	\begin{equation*}
	\rho v - \operatorname{div} ( K(x)\nabla v+b(x) \nabla v)=b_0 \hbox{ in } H^{-1}(\Omega).
	\end{equation*}
	Define,
	 \begin{equation}\label{3.12}
	\left\{\begin{aligned}
	w ={}& {}  u-p \in H_{0}^{1}(\Omega),\\
	z ={}& {} v-q \in H_{0}^{1}.
	\end{aligned}\right.
	\end{equation}
	Using equations \eqref{3.12} we deduce that $(u, v, w, z) \in D (A)$, since
	 \begin{equation}\label{3.13}
	\left\{\begin{aligned}
	\frac{1}{\rho} \operatorname{div}(K(x)\nabla u+a(x)\nabla w)={}&u-p-r \in L^2(\Omega),\\
	\frac{1}{\rho} \operatorname{div}(K(x)\nabla v+b(x)\nabla z)={}&v-q-s \in L^2(\Omega),\\
	\end{aligned}\right.
	\end{equation}
	which gives us the desired solution. By the Theorem $1.4.6$ in Pazy's book \cite{pazy} we obtain $\overline{D(\mathcal{A})}=\mathcal{H}$. Thus, $\mathcal{A}$ generates a $C_0$ semigroup of contractions on $H$ by Lumer-Phillips Theorem.

	Moreover the nonlinear operator $\mathcal{F}: \mathcal{H} \rightarrow \mathcal{H}$ given in \eqref{3.5} is a locally Lipschitz continuous operator.  Given a bounded set $B$ in $\mathcal{H}$, there exists $R>0$ such that $\|U\|\leq R$ for all $U \in B$. Given $U=(u_{1},u_{2},u_{3},u_{4})$ and $V=(v_{1},v_{2},v_{3},v_{4})$ elements of $B$, we have that
	\begin{align}\label{3.14}
		\|\mathcal{F}U - \mathcal{F}V\|_{\mathcal{H}}^{2}   = {} &  \left\|  \left(0, 0, -\frac{1}{\rho}(u_{2}^{2}u_{1}-v_{2}^{2}v_{1}), -\frac{1}{\rho}(u_{1}^{2}u_{2}-v_{1}^{2}v_2)\right)\right\|_{\mathcal{H}}^{2} \nonumber \\
		= {}\nonumber & \left\|\frac{1}{\rho}(v_{2}^{2}v_{1}-u_{2}^{2}u_{1})\right\|_{L^2(\Omega)}^{2}+\left\| \frac{1}{\rho}(v_{1}^{2}v_2-u_{1}^{2}u_{2})\right\|_{L^2(\Omega)}^{2} \nonumber \\
		= {} & \left\|\frac{1}{\rho}(v_{2}^{2}v_{1}-u_{2}^{2}v_{1}+u_{2}^{2}v_{1}-u_{2}^{2}u_{1})\right\|_{L^2(\Omega)}^{2}+\left\| \frac{1}{\rho}(v_{1}^{2}v_2-u_{1}^{2}v_{2}+u_{1}^{2}v_{2}-u_{1}^{2}u_{2})\right\|_{L^2(\Omega)}^{2} \nonumber \\
		=  {} & \left\|\frac{1}{\rho}((v_{2}^{2}-u_{2}^{2})v_{1}+u_{2}^{2}(v_{1}-u_{1}))\right\|_{L^2(\Omega)}^{2}+\left\| \frac{1}{\rho}((v_{1}^{2}-u_{1}^{2})v_{2}+u_{1}^{2}(v_{2}-u_{2}))\right\|_{L^2(\Omega)}^{2}\nonumber \\
		 \lesssim {} & \left\|\frac{1}{\rho}(v_{2}^{2}-u_{2}^{2})v_{1}\right\|_{L^2(\Omega)}^{2}+\left\| \frac{1}{\rho}(v_{1}^{2}-u_{1}^{2})v_{2}\right\|_{L^2(\Omega)}^{2} \nonumber \\
		& {}+{}\left\|\frac{1}{\rho}u_{2}^{2}(v_{1}-u_{1})\right\|_{L^2(\Omega)}^{2}+\left\|\frac{1}{\rho}u_{1}^{2}(v_{2}-u_{2})\right\|_{L^2(\Omega)}^{2}.
		\end{align}
	Let's show that each of the\eqref{3.14} is limited by a positive constant that depends on $R$ multiplied by $\|U-V\|_{\mathcal{H}}$.  For this, let $i, j \in \{1,2\}$ and consider $i \neq j$, then from the H�lder's inequality we obtain
	\begin{align}\label{3.15}
	\left\| \frac{1}{\rho}(v_{i}^{2}-u_{i}^{2})v_{j} \right\|_{L^2(\Omega)}^{2}	 = {} &  \int_{\Omega}\frac{1}{\rho^2}|v_i-u_i|^2|v_i+u_i|^2|v_j|^2 \, dx \nonumber \\
	\lesssim {}  & \left( \int_{\Omega} |v_i-u_i|^6 \, dx \right)^{\frac{1}{3}}\left(\int_{\Omega}|v_i+u_i|^3|v_{j}|^3 \, dx  \right)^{\frac{2}{3}} \nonumber  \\
	 \lesssim {}  & \left( \int_{\Omega} |v_i-u_i|^6 \, dx \right)^{\frac{1}{3}} \left[ \left( \int_{\Omega} |v_j|^6 \, dx \right)^{\frac{1}{2}}\left( \int_{\Omega}|v_i+u_i|^6  \, dx \right)^{\frac{1}{2}}  \right]^{\frac{2}{3}} \nonumber  \\
	 \lesssim {}  & \|v_j\|_{L^6(\Omega)}^2\|v_i+u_i\|_{L^6(\Omega)}^{2}\|v_i-u_i\|_{L^6(\Omega)}^{2}  \nonumber \\
	 \lesssim  {} & \|v_j\|_{L^6(\Omega)}^2 \left( \|v_i\|_{L^6(\Omega)}^2+\|u_i\|_{L^6(\Omega)}^2 \right)\|v_i-u_i\|_{L^6(\Omega)}^{2}  \nonumber \\
	 \lesssim {} & \|v_j\|_{H_{0}^{1}(\Omega)}^2 \left( \|v_i\|_{H_{0}^{1}(\Omega)}^2+\|u_i\|_{H_{0}^{1}(\Omega)}^2 \right)\|v_i-u_i\|_{H_{0}^{1}(\Omega)}^{2}  \nonumber \\
	 \lesssim {}  &  2R^3\|U-V\|_{\mathcal{H}}^{2}.
	\end{align}
    Proceeding in the same way, we obtain
    \begin{align}\label{3.16}
    \left\|\frac{1}{\rho}u_{j}^{2}(v_i-u_i)  \right\|_{L^2(\Omega)}^{2} = {} & \int_{\Omega}\frac{1}{\rho^2}|u_{j}|^{4}|v_i-u_i|^2\, dx \nonumber  \\
    \lesssim {}  & \left(\int_{\Omega} |u_j|^6 \, dx \right)^{\frac{2}{3}} \left(\int_{\Omega}|v_i-u_i|^6 \, dx  \right)^{\frac{1}{3}} \nonumber  \\
   \lesssim {}  & \|u_{j}\|_{L^6(\Omega)}^{4}\|v_i-u_i\|_{L^6(\Omega)}^{2} \nonumber  \\
    \lesssim {}  & \|u_{j}\|_{H_{0}^{1}(\Omega)}^{4}\|v_i-u_i\|_{H_{0}^{1}(\Omega)}^{2} \nonumber \\
    \lesssim {}  & R^4 \|v_i-u_i\|_{H_{0}^{1}(\Omega)}^{2} \nonumber \\
   \lesssim {}  & R^4 \|U-V\|_{\mathcal{H}}^{2}.
    \end{align}
    From \eqref{3.14}, \eqref{3.15} and \eqref{3.16} we obtain
    \begin{equation*}
    \|\mathcal{F}U-\mathcal{F}V\|_{\mathcal{H}}^{2} \leq C(R) \|U-V\|_{\mathcal{H}}^{2},
    \end{equation*}
    which proves that $\mathcal{F}$ is locally Lipschitz. Hence by Theorem $6.1.4$ and $6.1.5$ in Pazy's book \cite{pazy}, the Cauchy problem \eqref{1.1} has a unique mild solution
	\begin{equation*}
	W(t) = e^{\mathcal{A}t}W(0) + \int_{0}^{t} e^{\mathcal{A}(t-s)}\mathcal{F}(W(s))ds, \hbox{ for all } t \in [0,T_{\max}).
	\end{equation*}
	Let us see that $T_{\max} = \infty$. Indeed, given the energy functional defined in \eqref{1.5}, it follows that
	\begin{equation*}
	\frac{d}{dt}E_{u,v}(t) = -\int_{\Omega} a(x)|\nabla u_{t}(x,t)|^{2}dx-\int_{\Omega} b(x)|\nabla v_{t}(x,t)|^{2}dx,
	\end{equation*}
	which shows that $E_{u,v}(t)$ is non-increasing with $E_{u,v}(t) \leq E_{u}(0)$ for all $t \in [0, T_{\rm max})$. On the other hand, we have that
	\begin{align*}
	E_{u,v}(t) \geq {} & \frac{1}{2} \int_{\Omega} \rho(x)|u_t(x,t)|^{2}+\rho(x)|v_t(x,t)|^{2} \, dx\\
	& {} +{} \frac{1}{2} \int_{\Omega} \nabla u(x,t)^{\top} \cdot K(x) \cdot \nabla u(x,t)+\nabla v(x,t)^{\top} \cdot K(x) \cdot \nabla v(x,t) \, dx \\
	= {} & \frac{1}{2}\|(u,v,u_t,v_t )\|_{\mathcal{H}}^{2}.
	\end{align*}
	Thus,
	\begin{equation*}
	\frac{1}{2} \|(u, v, u_t, v_{t})\|_{\mathcal{H}}^{2} \leq E_{u}(t) \leq E_{u}(0),  \hbox{ for all } t \in [0, T_{\rm max}).
	\end{equation*}
	Therefore the local solutions cannot blow-up in finite time and it follows that $T_{\max}= \infty $.\end{proof}

\section{Exponential Stability}
\setcounter{equation}{0}

In this section we give the proofs of the main result of the present article which reads as follows
\begin{theorem}\label{theorem3} Under Assumptions presented above, given $R>0$, there exist constants $C$ and $\gamma$ such that the following inequality holds
	\begin{eqnarray}\label{4.1}
	E_{u,v}(t) \leq C e^{-\gamma t} E_{u,v}(0), ~t>0,
	\end{eqnarray}
	for all weak solution to problem \eqref{1.1} provided that $E_{u,v}(0)\leq R$.
\end{theorem}
\begin{remark}\label{remark3}	
	By standard density arguments it is enough to work with regular solutions  at   all times  since the decay rate estimate given in \eqref{4.1} can be recovered for weak solutions as well.
\end{remark}



In order to prove the Theorem \eqref{theorem3} and having in mind that problem \eqref{1.1} satisfies the semigroup property, so, in view of the identity of the energy associated to problem \eqref{1.1}, namely,
\begin{eqnarray}\label{4.2}
E_{u,v}(t_2) - E_{u,v}(t_1) &=& -\int_{t_1}^{t_2}\int_\Omega a(x) |\nabla u_t(x,t)|^2 \, dxdt \nonumber \\
&&-\int_{t_1}^{t_2}\int_\Omega b(x)|\nabla v_t(x,t)|^2\,dxdt, \hbox{ for all }0 \leq t_1 \leq t_2,
\end{eqnarray}
it is enough to prove that the following observability estimate holds:
\begin{lemma}\label{lemma4.1}
	For every $T > 0$ and every
	$R > 0$, there exists a constant $C > 0$ such that inequality
	\begin{eqnarray}\label{4.3}
	E_{u,v}(0) \leq C \int_0^T \int_\Omega a(x) |\nabla u_t(x,t)|^2+b(x)|\nabla v_t(x,t)|^2 \,dxdt,
	\end{eqnarray}
	holds for every regular solution $u$ of the damped problem \eqref{1.1} if the initial data satisfies
	\begin{eqnarray}\label{4.4}
	E_u(0)\leq R.
	\end{eqnarray}
\end{lemma}
\begin{proof}
	We argue by contradiction. Let us assume that \eqref{4.3} does not hold. Then, there exist a sequence $\{u^n,v^n\}$ of regular solutions to problem, according to Theorem \ref{theorem2}, such that the initial data satisfy
	\begin{eqnarray}\label{4.5}
	E_{u^n,v^n}(0)\leq L, \hbox{ for all } n\in \mathbb{N}, \hbox{ and for some }L>0,
	\end{eqnarray}
	and, in addition,
	\begin{eqnarray}\label{4.6}
	\lim_{n\rightarrow \infty} \frac{E_{u^n,v^n}(0)}{\int_0^T \int_\Omega a(x) |\nabla u_t^n(x,t)|^2+b(x) |\nabla v_t^n(x,t)|^2\,dxdt}=+\infty.
	\end{eqnarray}
	
	From \eqref{4.6} yields
	\begin{eqnarray}\label{4.7}
	\lim_{n\rightarrow \infty} \frac{\int_0^T \int_\Omega a(x) |\nabla u_t^n(x,t)|^2+b(x) |\nabla v_t^n(x,t)|^2\,dxdt}{E_{u^n,v^n}(0)}=0.
	\end{eqnarray}
	
	Taking \eqref{4.5} and \eqref{4.7} into account, we deduce
	\begin{eqnarray}\label{4.8}
	\lim_{n\rightarrow \infty}\int_0^T \int_\Omega a(x) |\nabla u_t^n(x,t)|^2+b(x) |\nabla v_t^n(x,t)|^2\,dxdt=0,
	\end{eqnarray}
	and since $a(x) \geq a_0 >0$ a.e. in $\omega$ from \eqref{4.8} we infer
	\begin{eqnarray}\label{4.9}
	\lim_{n\rightarrow \infty}\int_0^T \int_{\omega} |\nabla u_t^n(x,t)|^2\,dxdt=0,
	\end{eqnarray}
	and since that $b(x)\geq b_0 >0$ a.e. in $\Omega$ we obtain that
		\begin{eqnarray}\label{4.10}
	\lim_{n\rightarrow \infty}\int_0^T \int_{\omega}  |\nabla v_t^n(x,t)|^2\,dxdt=0.
	\end{eqnarray}
	From a variant of Poincar\'e's inequality, from \eqref{4.9} and \eqref{4.10}, it follows that
	\begin{eqnarray}\label{4.11}
	\lim_{n\rightarrow \infty}\int_0^T \int_{\omega}  |u_t^n(x,t)|^2\,dxdt=0
	\end{eqnarray}
	and
	\begin{eqnarray}\label{4.12}
	\lim_{n\rightarrow \infty}\int_0^T \int_{\omega}  |v_t^n(x,t)|^2\,dxdt=0.
	\end{eqnarray}
	
	Since $E_{u^,v^n}(t)$ is non-increasing and $E_{u^n,v^n}(0)$ remains bounded then, we obtain a subsequence, still denoted by $\{(u^n,v^n)\}$ which verifies
	\begin{eqnarray}\label{4.13}
	&&(u^n,v^n) \overset{\ast}{\rightharpoonup} (u,v) \hbox{ in } (L^\infty(0,T;H_0^1(\Omega)))^2,\\
	&&(u_t^n,u_{t}^{n}) \overset{\ast}{\rightharpoonup} (u_t,v_t) \hbox{ in } (L^\infty(0,T;L^2(\Omega)))^2. \label{4.14}
	\end{eqnarray}
	
	From standard compactness arguments, see \cite{Lions2} or \cite{Simon}, we deduce, for an eventual subsequence, that will be denote by the same notation, that
	\begin{eqnarray}\label{4.15}
	&& (u^n,v^n) \rightarrow (u,v) \hbox{ in } (L^\infty(0,T;L^q(\Omega)))^2, \hbox{ for all } q\in[2,6)
	\end{eqnarray}

	From \eqref{4.15} we infer
	\begin{equation*}
	(v^n)^2u^n \rightarrow v^2u \hbox{ a.e. in } \Omega \times (0,T)
	\end{equation*}
	and
	\begin{equation*}
	(u^n)^2v^n \rightarrow u^2v \hbox{ a.e. in } \Omega \times (0,T).
	\end{equation*}

	Moreover, $((v^n)^2u^n)$ and $((u^n)^2v^n)$ are bounded in $L^2(0,T;L^2(\Omega))$. Indeed,
    \begin{align*}
	\int_{0}^{T} \int_{\Omega} |(v^n)^2u^n|^2 \, dxdt = {} & \int_{0}^{T} \int_{\Omega} |v^n|^4|u^n|^2 \, dx dt \\
	\leq {}& \int_{0}^{T} \left(\int_{\Omega} |v^n|^6 \, dx\right)^{\frac{2}{3}} \left(\int_{\Omega}|u^n|  \, dx\right)^{\frac{1}{3}} \, dt\\
	 = {} & \int_{0}^{T} \|v^n\|_{L^6(\Omega)}^{4} \|u^n\|_{L^6(\Omega)}^{2} \, dt \\
	 \leq {} & \|v^n\|_{L^{8}(0,T;L^6(\Omega))}^{4} \|v^n\|_{L^{4}(0,T;L^6(\Omega))}^{2}\\
	 \lesssim{} & \|v^n\|_{L^\infty(0,T;H_{0}^{1}(\Omega))}\|u^n\|_{L^\infty(0,T;H_{0}^{1}(\Omega))}^{2} <\infty.
    \end{align*}
    Analogously we have the another limitation. From the Lions Lemma we conclude that
    \begin{equation}\label{4.16}
    (v^n)^2u^n \rightharpoonup v^2u \hbox{ and } (u^n)^2v^n \rightharpoonup u^2v \hbox{ in } L^2(0,T;L^2(\Omega)).
    \end{equation}
	
	In this point we shall divide our proof into two cases: (i) ~ $u\ne0$ and $v \ne 0$  and  (ii) $(u,v)=(0,0)$. We note that these are the only two cases we should consider. Indeed, if $u = 0$ then necessarily $v = 0$ because if we take the following sequence of problems into account
	\begin{equation}\label{4.17}
	\begin{cases}
	\rho(x) u_{tt}^{n} - \operatorname{div}(K(x) \nabla u^{n}) + (v^{n})^2u^n- \operatorname{div}(a(x) \nabla  u_t^n)=0  &~\hbox{ in }~ \Omega \times (0,T), \\
	\rho(x) v_{tt}^{n} - \operatorname{div}(K(x) \nabla v^{n}) + (u^{n})^2v^n - \operatorname{div}(b(x) \nabla v_t^n)=0  &~\hbox{ in }~ \Omega\times (0,T), \\
	u = v =  0 &~\hbox{ on }~\partial \Omega \times (0,T),\\
	u^n(x,0)=u_{0}^{n}(x), ~ u_t(x,0)=u_1^n(x) &~\hbox{ in }~ \Omega,\\
	v^n(x,0)=v_0^n(x), ~ v_t^n(x,0)=v_1^n(x) &~\hbox{ in }~ \Omega,
	\end{cases}
	\end{equation}
	and passing to the limit in \eqref{4.17} taking \eqref{4.9}, \eqref{4.10} and \eqref{4.13}-\eqref{4.16}, into account, we obtain, in the distributional sense,
	\begin{equation}\label{4.18}
	\begin{cases}
	\rho(x) v_{tt} - \operatorname{div}(K(x) \nabla v)=0  &~\hbox{ in }~ \Omega\times (0,T), \\
	v_t=0 &~\hbox{ in }~ \omega\times (0,T),
	\end{cases}
	\end{equation}
	and for $z=v_t$ from \eqref{4.18} we infer
	\begin{equation}\label{4.19}
	\begin{cases}
	\rho(x) z_{tt} - \operatorname{div}(K(x) \nabla z)=0  &~\hbox{ in }~ \Omega\times (0,T), \\
	z=0 &~\hbox{ in }~ \omega\times (0,T),
	\end{cases}
	\end{equation}
	which implies that $z=v_t=0$ and from \eqref{4.18} we deduce that $v=0$. Analogously if $v=0$ then $u=0$.
	
	\noindent{Case (i):~$u\ne 0$ and $v\ne 0$.}

	Taking the following subsequence of problems into account
	\begin{equation}\label{4.20}
	\begin{cases}
	\rho(x) u_{tt}^{n} - \operatorname{div}(K(x) \nabla u^{n}) + (v^{n})^2u^n- \operatorname{div}(a(x) \nabla  u_t^n)=0  &~\hbox{ in }~ \Omega \times (0,T), \\
	\rho(x) v_{tt}^{n} - \operatorname{div}(K(x) \nabla v^{n}) + (u^{n})^2v^n - \operatorname{div}(b(x) \nabla v_t^n)=0  &~\hbox{ in }~ \Omega\times (0,T), \\
	u^n(x,0)=u_{0}^{n}(x), ~ u_t(x,0)=u_1^n(x) &~\hbox{ in }~ \Omega,\\
	v^n(x,0)=v_0^n(x), ~ v_t^n(x,0)=v_1^n(x) &~\hbox{ in }~ \Omega,
	\end{cases}
	\end{equation}
	and passing to the limit in \eqref{4.20} taking \eqref{4.9}, \eqref{4.10} and \eqref{4.13}-\eqref{4.16}, into account, we obtain, in the distributional sense,
	\begin{equation}\label{4.21}
	\begin{cases}
	\rho(x) u_{tt} - \operatorname{div}(K(x) \nabla u) + v^2u =0  &~\hbox{ in }~ \Omega \times (0,T), \\
	\rho(x) v_{tt} - \operatorname{div}(K(x) \nabla v) + u^2v=0  &~\hbox{ in }~ \Omega\times (0,T), \\
	u_t=v_t=0 &~\hbox{ in }~ \omega\times(0,T)
	\end{cases}
	\end{equation}
	and for $w=u_t$ and $z=v_t$ from \eqref{4.21} we infer
	\begin{equation}\label{4.22}
	\begin{cases}
	\rho(x) w_{tt} - \operatorname{div}(K(x) \nabla w) + 2vuz+v^2w =0  &~\hbox{ in }~ \Omega \times (0,T), \\
	\rho(x) z_{tt} - \operatorname{div}(K(x) \nabla z) + 2uvw+u^2z=0  &~\hbox{ in }~ \Omega\times (0,T), \\
	w=z=0 &~\hbox{ in }~ \omega\times(0,T),
	\end{cases}
	\end{equation}
	Defining $V_1(,x,t)=v^2$, $V_2(x,t)=u^2$ and $V_3(x,t)=-2uv$, the system \eqref{4.22} can be rewrite like
	\begin{equation}\label{4.23}
	\begin{cases}
	\rho(x) w_{tt} - \operatorname{div}(K(x) \nabla w) + V_1(x,t)w =V_3(x,t)z  &~\hbox{ in }~ \Omega \times (0,T), \\
	\rho(x) z_{tt} - \operatorname{div}(K(x) \nabla z) + V_2(x,t)z=V_3(x,t)w  &~\hbox{ in }~ \Omega\times (0,T), \\
	w=z=0 &~\hbox{ in }~ \omega\times(0,T),
	\end{cases}
	\end{equation}
	Employing Assumption \ref{assumption1.3}, we conclude that $(w,z)=(0,0)$, and, consequently from \eqref{4.21} it follows that $u=v=0$, which is a contradiction.
	
	\noindent{Case (ii):~$(u,v)=(0,0)$.}
		
	Now, we define:
	\begin{eqnarray}\label{4.24}
	\alpha_n:= [E_{u^n,v^n}(0)]^{1/2},~ w^n:= \frac{u^n}{\alpha_n} \hbox{ and } z^n:= \frac{v^n}{\alpha_n}.
	\end{eqnarray}
		Now, let us consider the following subsequence of problems
	\begin{equation}\label{4.25}
	\begin{cases}
	\rho(x) w_{tt}^{n} - \operatorname{div}(K(x) \nabla w^{n}) + \alpha_n^2(z^{n})^2w^n- \operatorname{div}(a(x) \nabla  w_t^n)=0  &~\hbox{ in }~ \Omega \times (0,T), \\
	\rho(x) z_{tt}^{n} - \operatorname{div}(K(x) \nabla z^{n}) + \alpha_{n}^{2}(w^{n})^2z^n - \operatorname{div}(b(x) \nabla \ z_t^n)=0  &~\hbox{ in }~ \Omega\times (0,T), \\
	w_n = z_n =  0 &~\hbox{ on }~\partial \Omega \times (0,T),\\
	w^n(x,0)=w_{0}^{n}(x)=\frac{u_0^n}{\alpha_n}, ~ w_t(x,0)=w_1^n(x)=\frac{u_1^n}{\alpha_n} &~\hbox{ in }~ \Omega,\\
	z^n(x,0)=z_0^n(x)=\frac{v_0^n}{\alpha_n}, ~ z_t^n(x,0)=z_1^n(x)=\frac{v_1^n}{\alpha_n} &~\hbox{ in }~ \Omega.
	\end{cases}
	\end{equation}
	A simple calculation shows that
	\begin{equation}\label{4.26}
	E_{w^n,z^n}(t)=\frac{1}{\alpha_n^2}E_{u^n,v^n}(t).
	\end{equation}
	It is not difficult to check that there exists a positive constant $C$ such that $1/C\leq E_{w_n,z^n}(0)=1$ for all $n\in \mathbb{N}$. In order to achieve a contradiction we need to prove that  $E_{w^n,z^n}(0)$ converges to zero. Indeed, first, taking \eqref{4.8} and \eqref{4.24} into consideration we infer
	\begin{eqnarray}\label{4.27}
	\lim_{n\rightarrow \infty}\int_0^T \int_\Omega a(x) |\nabla w_t^n(x,t)|^2\,dxdt=0,
	\end{eqnarray}
	and
	\begin{eqnarray}\label{4.28}
	\lim_{n\rightarrow \infty}\int_0^T \int_\Omega a(x) |\nabla z_t^n(x,t)|^2\,dxdt=0.
	\end{eqnarray}
	Since $a(x) \geq a_0 >0$ a.e. in $\omega$ from \eqref{4.27} we infer
	\begin{eqnarray}\label{4.29}
	\lim_{n\rightarrow \infty}\int_0^T \int_{\omega} |\nabla w_t^n(x,t)|^2\,dxdt=0,
	\end{eqnarray}
	and since that $b(x)\geq b_0 >0$ a.e. in $\omega$  from \eqref{4.28} we obtain that
	\begin{eqnarray}\label{4.30}
	\lim_{n\rightarrow \infty}\int_0^T \int_{\omega}  |\nabla z_t^n(x,t)|^2\,dxdt=0.
	\end{eqnarray}
	Taking \eqref{4.29} and \eqref{4.30} into account and considering a variant of Poincar\'e's inequality, it follows that
	\begin{eqnarray}\label{4.31}
	\lim_{n\rightarrow \infty}\int_0^T \int_{\omega}  |w_t^n(x,t)|^2\,dxdt=0
	\end{eqnarray}
	and
	\begin{eqnarray}\label{4.32}
	\lim_{n\rightarrow \infty}\int_0^T \int_{\omega}  |z_t^n(x,t)|^2\,dxdt=0.
	\end{eqnarray}
	
	Further, from the boundness  $E_{w^n,z^n}(0)\leq C$ we also deduce that for an eventual subsequence of $\{w^n,z^n\}$ that
	\begin{eqnarray}\label{4.33}
    &&(w^n,z^n) \overset{\ast}{\rightharpoonup} (w,z) \hbox{ in } (L^\infty(0,T;H_0^1(\Omega)))^2,\\
    &&(w_t^n,z_{t}^{n}) \overset{\ast}{\rightharpoonup} (w_t,z_t) \hbox{ in } (L^\infty(0,T;L^2(\Omega)))^2.\label{4.34}
	\end{eqnarray}
	From standard compactness arguments, we deduce, for an eventual subsequence, that will be denote by the same notation, that
	\begin{eqnarray}\label{4.35}
	&& (w^n,z^n) \rightarrow (w,z) \hbox{ in } (L^\infty(0,T;L^q(\Omega)))^2, \hbox{ for all } q\in[2,6)
	\end{eqnarray}

	Note that, for an eventual subsequence, $\alpha_n \rightarrow \alpha \in [0,+\infty)$. If $\alpha = 0$ and since
		\begin{eqnarray}
	&&\alpha_nw^n =u^n \rightarrow 0 \hbox{ in } L^\infty(0,T;H_0^1(\Omega)))^2, \label{4.36}\\
	&&\alpha_nz^n =v^n \rightarrow 0 \hbox{ in } L^\infty(0,T;H_0^1(\Omega)). \label{4.37}
	\end{eqnarray}
	Taking \eqref{4.27}, \eqref{4.28}, \eqref{4.33}, \eqref{4.34}, \eqref{4.36} and \eqref{4.37} into account, passing to the limit in \eqref{4.25} arrive at
	\begin{equation}\label{4.38}
	\begin{cases}
	\rho(x) w_{tt} - \operatorname{div}(K(x) \nabla w)  =0  &~\hbox{ in }~ \Omega \times (0,T), \\
	\rho(x) z_{tt} - \operatorname{div}(K(x) \nabla z)  =0  &~\hbox{ in }~ \Omega\times (0,T), \\
	w_t=z_t=0 &~\hbox{ in }~ \omega\times(0,T).
	\end{cases}
	\end{equation}
	For $\varphi=w_t$ and  $\psi=z_t$ from \eqref{4.38} we infer
    \begin{equation}\label{4.39}
    \begin{cases}
    \rho(x) \varphi_{tt} - \operatorname{div}(K(x) \nabla \varphi)  =0  &~\hbox{ in }~ \Omega \times (0,T), \\
    \rho(x) \psi_{tt} - \operatorname{div}(K(x) \nabla \psi)  =0  &~\hbox{ in }~ \Omega\times (0,T), \\
    \varphi=\psi=0 &~\hbox{ in }~ \omega\times(0,T).
    \end{cases}
    \end{equation}
	which implies that $\varphi=\psi=0$ and, consequently from \eqref{4.38}, $w=z=0$.
	
	Now, let us consider $\alpha>0$. So, passing to the limit in \eqref{4.25} we arrive at
	\begin{equation}\label{4.40}
	\begin{cases}
	\rho(x) w_{tt} - \operatorname{div}(K(x) \nabla w) + \alpha^2 z^2w =0  &~\hbox{ in }~ \Omega \times (0,T), \\
	\rho(x) z_{tt} - \operatorname{div}(K(x) \nabla z) + \alpha^2w^2z=0  &~\hbox{ in }~ \Omega\times (0,T), \\
	w_t=z_t=0 &~\hbox{ in }~ \omega\times(0,T).
	\end{cases}
	\end{equation}
	For $\varphi=w_t$ and $\psi=z_t$ from \eqref{4.40} we infer
	\begin{equation}\label{4.41}
	\begin{cases}
	\rho(x) \varphi_{tt} - \operatorname{div}(K(x) \nabla w) + 2\alpha^2 zw\psi +\alpha^2z^2\varphi=0  &~\hbox{ in }~ \Omega \times (0,T), \\
	\rho(x) \psi_{tt} - \operatorname{div}(K(x) \nabla \psi) + 2\alpha^2 wz\varphi+\alpha^2w^2\psi=0  &~\hbox{ in }~ \Omega\times (0,T), \\
	\varphi=\psi=0 &~\hbox{ in }~ \omega\times(0,T).
	\end{cases}
	\end{equation}
	Denoting $V_1(x,t)=\alpha^2z^2$ , $V_2(x,t)=\alpha^2w^2$ and $V_3(x,t)=-2\alpha^2zw$ we can rewrite \eqref{4.41} like
		\begin{equation}\label{4.42}
	\begin{cases}
	\rho(x) \varphi_{tt} - \operatorname{div}(K(x) \nabla w) + V_1(x,t)\varphi=V_3(x,t)\psi  &~\hbox{ in }~ \Omega \times (0,T), \\
	\rho(x) \psi_{tt} - \operatorname{div}(K(x) \nabla \psi) + V_2(x,t)\psi=V_3(x,t)\varphi  &~\hbox{ in }~ \Omega\times (0,T), \\
	\varphi=\psi=0 &~\hbox{ in }~ \omega\times(0,T).
	\end{cases}
	\end{equation}
	Employing Assumption \ref{assumption1.3}, from \eqref{4.42} it follows that $(\varphi,\psi)=(0,0)$ and, consequently, $w=z=0$. Thus in any of both cases $w=z=0$. As a consequence $w=z=0$ in all the convergences in \eqref{4.33}-\eqref{4.35}.
	
	Proceeding in the same way as in \eqref{4.16}, and using the fact that $w=z=0$, we obtain
	\begin{equation}\label{4.43}
	\alpha_n^2(z^n)^2w^n \rightharpoonup 0 \hbox{ and } \alpha_n^2(w^n)^2z^n \rightharpoonup 0 \hbox{ in } L^2(0,T;L^2(\Omega)).
	\end{equation}
	Remember that our main objective is to prove that $E_{w^n,z^n}(0)$ converges to zero,  where
		\begin{eqnarray*}
		E_{w^n,z^n}(t) &=& \frac{1}{2} \int_{\Omega} \rho|w_t^n(x,t)|^2 +\rho|z_t^n(x,t)|^2 +(\nabla w^{n}(x,t))^{\top}\cdot K(x) \cdot \nabla w^n(x,t)\, dx \\
		&&\frac{1}{2} \int_{\Omega} (\nabla z^{n}(x,t))^{\top}\cdot K(x) \cdot \nabla z^n(x,t) +\alpha_n^2(w^nz^n)^2 \, dx
		\end{eqnarray*}

		Let us use the following notation
		\begin{equation*}
		P := -\rho \partial_{t}^{2} - \sum_{i=1}^{d} \partial_{x_{i}}(K(x)\partial_{x_{j}}).
		\end{equation*}
		Then, we have that given the weak convergence in $L^{2}(\Omega \times (0,T))$ of the nonlinear terms $\alpha_{n}^{2}(z^n)^2w^n$ and $\alpha_{n}^{2}(w^n)^2z^n$, we deduce that
		\begin{equation}\label{4.44}
		Pw_{t}^{n} \rightarrow 0 \hbox{ in } H^{-1}_{loc}(\Omega \times (0,T))
		\end{equation}
		and
		\begin{equation}\label{4.45}
		Pz_{t}^{n} \rightarrow 0 \hbox{ in } H^{-1}_{loc}(\Omega \times (0,T)).
		\end{equation}
		
		Let us also denote by $\mu_1$ and $\mu_2$ the microlocal defect measures associated to $\{w_{t}^n\}$ and $\{z_t^n\}$ in $L^2(\Omega \times (0,T))$, which is assured by Theorem \ref{theorem4}. Then, taking into account Assumption \ref{assumption1.2}, we deduce two facts:
		
	(i) The supports of the measures $\mu_1$ and $\mu_2$ are contained in the characteristic set of the wave operator $\{\tau^2-\frac{1}{\rho}\xi^{\top} \cdot K(x)\cdot \xi=0\}$.
	
	(ii) $\mu_1$ and $\mu_2$ propagates along the bicharacteristic  flow of this operator, which signifies, particularly, that if some point $\omega_0=(x_0,t_0,\xi_0,\tau_0)$ does not belong to the $\supp(\mu_1)$ or $\supp (\mu_2)$ the whole bicharacteristic issued from $\omega_0$ is out of $\supp(\mu_1)$ or $\supp (\mu_2)$.

		Indeed, from \eqref{4.44}, \eqref{4.45} and Theorem \ref{theorem5} we deduce item $(i)$.
				
		Furthermore, from Proposition \ref{proposition5.1} and Theorem \ref{theorem7} found in the Appendix, we deduce that $\supp(\mu)$ in $(\Omega \times (0,T))\times S^d$ is a union of curves like
		\begin{eqnarray}\label{4.46}
		t \in I\cap (0,\infty) \mapsto m\pm(t)=\left(t, x(t), \frac{\pm1}{\sqrt{1+|G(x)\dot{x}|}},\frac{\pm G(x) \dot{x}}{\sqrt{1+|G(x)\dot{x}|}} \right),
		\end{eqnarray}
		where $t\in I \mapsto x(t)\in \Omega$ is a geodesic associated to the metric .
		
		Since $ w_t^{n} \rightarrow 0$  and $ z_t^{n} \rightarrow 0$ in $L^2(\omega \times (0,T))$, we deduce that $\mu_1=\mu_2=0$ in $\omega$ and consequently $\supp(\mu_1) \subset (\Omega \setminus \omega) \times (0,T)$ and $\supp(\mu_2) \subset (\Omega \setminus \omega) \times (0,T)$.
		
		On the other hand, let $t_0 \in (0,+\infty)$ and let $x$ be a geodesic of $G$ defined near $t_0$. Once the geodesics inside $\Omega\backslash\omega$, enter necessarily in the region $\omega$, then, for any geodesic of the metric $G$, with $0\in I$ there exists $t >0 $ such that $m\pm(t)$ does not belong to the $\supp(\mu_1)$ and $\supp (\mu_2)$, so that $m\pm(t_0)$ does not belong as well and item $(ii)$ follows.  Once the time $t_0$ and the geodesic $x$ were taken arbitrary, we conclude that $\supp(\mu_1)$ and $\supp (\mu_2)$ are empty. As a consequence, $w_t^n\rightarrow 0$  and $z_t^n\rightarrow 0$ in $L^2_{loc}(\Omega \times (0,T))$. Since, $v_t^n \rightarrow 0$ and $z_t^n \rightarrow 0$ strongly in $L^2(\omega \times (0,T))$ we deduce
		\begin{eqnarray}\label{4.47}
		w_t^n \rightarrow 0 \hbox{ in } L^2(0,T;L^2(\Omega))
		\end{eqnarray}
		and
		\begin{eqnarray}\label{4.48}
		z_t^n \rightarrow 0 \hbox{ in } L^2(0,T;L^2(\Omega)).
		\end{eqnarray}
		It is then easy to show that $E_{w^n,z^n}(0)$ converges to zero. Indeed, let us consider the following cut-off function
		\begin{align*}
		&\theta\in C^{\infty}(0,T),  \quad    0\leq \theta(t) \leq 1,  \quad   \theta(t)=1 \ \mbox{in} \ (\varepsilon,T-\varepsilon).
		\end{align*}
		
		Multiplying  the first equation in \eqref{4.25} by $w^n \theta$, the second by $z^n \theta$ and integrating by parts, we infer
		\begin{eqnarray}\label{4.49}
		&& -\int_0^T \theta(t)\int_\Omega \rho(x)|w_t^n|^2\,dxdt - \int_0^T \theta'(t)\int_\Omega\rho(x) w_t^n w^n\,dxdt\\
		&&+\int_0^T \theta(t) \int_\Omega (\nabla w^n)^{\top} \cdot K(x) \cdot \nabla w^n \,dxdt + \alpha_{n}^{2}\int_0^T \theta(t) \int_\Omega  (z^n)^2(w^n)^2 \,dxdt\nonumber\\
		&&+ \int_0^T \theta(t) \int_\Omega a(x) \nabla w_t^n \cdot \nabla w^n\,dxdt=0 \nonumber
		\end{eqnarray}
		and
		\begin{eqnarray}\label{4.50}
		&& -\int_0^T \theta(t)\int_\Omega \rho(x)|z_t^n|^2\,dxdt - \int_0^T \theta'(t)\int_\Omega\rho(x) z_t^n z^n\,dxdt\\
		&&+\int_0^T \theta(t) \int_\Omega (\nabla z^n)^{\top} \cdot K(x) \cdot \nabla z^n \,dxdt + \alpha_{n}^{2}\int_0^T \theta(t) \int_\Omega  (z^n)^2(w^n)^2 \,dxdt\nonumber\\
		&&+ \int_0^T \theta(t) \int_\Omega a(x) \nabla z_t^n \cdot \nabla z^n\,dxdt=0 \nonumber
		\end{eqnarray}
		Considering the convergences \eqref{4.27}, \eqref{4.28}, \eqref{4.33}-\eqref{4.35} and \eqref{4.43} having in mind that $w=z=0$ from \eqref{4.49}  and \eqref{4.50} we deduce that
		\begin{equation*}
		\lim_{n \rightarrow +\infty}\int_{0}^{T} \theta(t)\int_\Omega (\nabla w^{n})^{T} \cdot K(x) \cdot \nabla w^n  \,dxdt = 0 \hbox{ and } \lim_{n \rightarrow +\infty}\int_{0}^{T}\theta(t) \int_\Omega (\nabla z^{n})^{T} \cdot K(x) \cdot \nabla z^n  \,dxdt = 0,
		\end{equation*}
		which implies that
				\begin{equation*}
		\lim_{n \rightarrow +\infty}\int_\varepsilon^{T-\varepsilon} \int_\Omega (\nabla w^{n})^{T} \cdot K(x) \cdot \nabla w^n  \,dxdt = 0 \hbox{ and } \lim_{n \rightarrow +\infty}\int_\varepsilon^{T-\varepsilon} \int_\Omega (\nabla z^{n})^{T} \cdot K(x) \cdot \nabla z^n  \,dxdt = 0.
		\end{equation*}
		We also have
		\begin{equation*}
		\alpha_n^2\int_{\varepsilon}^{T-\varepsilon}\int_{\Omega}(w^n z^n)^2 \, dx dt \rightarrow 0.
		\end{equation*}
	
		Combining the above convergences we have that
		\begin{equation*}
		\int_{\varepsilon}^{T-\varepsilon}E_{w^n,z^n}(t) \, dt \rightarrow 0.
		\end{equation*}
		Then by the decrease of the energy, we obtain
		\begin{equation*}
		(T-2\varepsilon)E_{w^n,z^n}(T-2\varepsilon) \rightarrow 0.
		\end{equation*}	
		Combining the energy identity
		\begin{equation*}
		E_{w^n,z^n}(T-\varepsilon)-E_{w^n,z^n}(\varepsilon)=-\int_{\varepsilon}^{T-\varepsilon}a(x)|\nabla w_t^n(x,t)|^2+b(x)|\nabla z_t^n(x,t)|^2 \, dxdt,
		\end{equation*}
		with \eqref{4.27}, \eqref{4.28} and the arbitrariness of $\varepsilon>0$, it follows that $E_{w^n,z^n}(0) \rightarrow 0$ as $n\rightarrow +\infty$, as we desired to prove.
\end{proof}

\section{Appendix}
\setcounter{equation}{0}
\subsection{Part I}

We start this section by presenting some classical tools regarding the Hamiltonian vector field and its bicharacteristic curves in the cotangent $(x,\xi)-$space for real functions $p(x,\xi)$.

Let $p \in C^{\infty}(\Omega \times {\mathbb{R}}^d\backslash\{0\})$ be a real valued function.  We call $H_p$ an Hamiltonian field of $p$, the following vector field defined in $\Omega \times {\mathbb{R}}^d\backslash\{0\}$:
$$H_p(x,\xi)= \left(\frac {\partial p} { \partial \xi_1} (x,\xi), \cdots, \frac {\partial p} {\partial \xi_n}  (x,\xi); -\frac {\partial p} {\partial x_1} (x,\xi), \cdots, -\frac {\partial p} {\partial x_n} (x,\xi)\right).$$

The Lie derivative of a function $f$ with respect the Hamiltonian field $H_p$ is given by $H_p (f)=\{p,f\}$, where
\begin{eqnarray*}
	\{p,f\}(x,\xi)=\sum_{j=1}^d \left(\frac{\partial p}{\partial \xi_j}\frac{\partial f}{\partial x_j}-\frac{\partial p}{\partial x_j}\frac{\partial f}{\partial \xi_j}\right).
\end{eqnarray*}

An Hamiltonian curve of $p$ is an integrable curve of the vector field $H_p$, that is a maximal solution $s\in I \mapsto (x(s),~\xi(s))$ to the Hamilton-Jacobi equations
\begin{equation}\label{5.1}
\left\{
\begin{aligned}
& \dot{x}=p_\xi(x,\xi)=\frac{\partial p}{\partial \xi}(x,\xi),\quad \dot{\xi}=-p_x(x,\xi)=- \frac{\partial p}{\partial x},\\
\end{aligned}
\right.
\end{equation}
where $I$ is an open interval of $\mathbb{R}$.

Since $H_pp =0$ the function $p$ keeps a constant length value on each of its Hamiltonian curves.  We say that such curve is a bicharacteristic of $p$ if this value is null.

\begin{remark}\label{remark4}
	Let $\lambda$ be a $C^\infty$ function on $T^0\Omega$ with real values different of zero. Since
	\begin{eqnarray*}
		H_{\lambda p} = \lambda H_p + p H_\lambda = \lambda H_p ~\hbox{ if }p=0,
	\end{eqnarray*}
	it results that the bicharacteristics of $\lambda p$ and $p$ coincide (modulo a re-parametrization).
\end{remark}

Next, for the readers comprehension and for the sake of clarify, we shall announce some results due to Burq and G\'erard \cite{Burq-Gerard} (without their proofs) which were developed firstly in G\'erard \cite{Gerard}.

We have the first important result:

\begin{theorem}\label{theorem4}
	Let $\{u_n\}_{n\in \mathbb{N}}$ be  a bounded sequence in $L_{loc}^2(\Omega)$ such that it converges weakly to zero in $L_{loc}^2(\Omega)$. Then, there exists a subsequence $\{u_{\varphi(n)}\}$ and a positive Radon measure $\mu$ on $T^1\Omega:=\Omega\times S^{d-1}$ such that for all pseudo-differential operator $A$ of order $0$ on $\Omega$ which admits a principal symbol $\sigma_0(A)$ and for all $\chi \in C_0^\infty(\Omega)$ such that $\chi \sigma_0(A)=\sigma_0(A)$, one has
	\begin{eqnarray}\label{5.2}
	\left(A(\chi u_{\varphi(n)}), \chi u_{\varphi_n}\right)_{L^2}\underset{n\rightarrow +\infty}\longrightarrow \int_{\Omega \times S^{n-1}}\sigma_0(A)(x,\xi)\,d\mu(x,\xi).
	\end{eqnarray}
\end{theorem}

\begin{definition}\label{definition5.1}
	Under the circumstances of Theorem \ref{theorem4} $\mu$ is called the {\bf microlocal defect measure} of the sequence $\{u_{\varphi(n)}\}_{n\in \mathbb{N}}$.
\end{definition}

\begin{remark}\label{remark5}
	Theorem \ref{theorem4} assures that for all bounded sequence $\{u_n\}_{n\in \mathbb{N}}$ of $L_{loc}^2(\Omega)$ which converges weakly to zero, the existence of a subsequence admitting a microlocal defect measure (in short, m.d.m). We observe that from \eqref{5.2} in the particular case when $A=f\in C_0^\infty(\Omega)$, it follows that
	\begin{eqnarray}\label{5.3}
	\int_\Omega f(x) |u_{\varphi(n)}(x)|^2\,dx \rightarrow \int_{\Omega \times S^{d-1}}f(x)\,d\mu(x,\xi),
	\end{eqnarray}
	so that $u_{\varphi(n)}$ converges to $0$ strongly  if and only if $\mu=0$.
\end{remark}

The second important result reads as follows:

\begin{theorem}\label{theorem5}
	Let $P$ be a differential operator of order $m$ on $\Omega$ and let $\{u_n\}$ a bounded sequence of $L_{loc}^2(\Omega)$ which converges weakly to $0$ and admits a m.d.m. $\mu$. The following statement are equivalents:
	\begin{eqnarray*}
		&&(i)~Pu_n \underset{n \rightarrow +\infty}\longrightarrow 0 \hbox{ strongly in }H_{loc}^{-m}(\Omega)~(m>0). \ \ \ \ \ \ \ \ \ \ \ \ \ \ \ \ \ \ \ \ \  \ \ \ \ \ \ \ \ \ \ \ \ \ \ \ \ \ \ \ \ \ \\
		&&(ii) \hbox{supp} (\mu) \subset \{(x,\xi)\in \Omega \times S^{d-1}: \sigma_m(P)(x,\xi)=0\}. \ \ \ \ \ \ \ \ \ \ \ \ \ \ \ \ \ \ \ \ \ \ \ \ \ \ \ \ \ \ \ \ \ \ \ \ \ \ \ \ \ \
	\end{eqnarray*}
\end{theorem}

\begin{theorem}\label{theorem6}
	Let $P$ be a differential operator of order $m$ on $\Omega$, verifying $P^\ast=P$, and let $\{u_n\}$ be a bounded sequence in $L_{loc}^2(\Omega)$ which converges weakly to $0$ and it admits a m.d.m. $\mu$. Let us assume that $P u_n \underset{n\rightarrow +\infty}\longrightarrow 0$ strongly in $H_{loc}^{1-m}$. Then, for all function $a\in C^\infty(\Omega \times (\mathbb{R}^d)\backslash\{0\})$ homogeneous of degree $1-m$ in the second variable and with compact support in the first one,
	\begin{eqnarray}\label{5.4}
	\int_{\Omega \times s^{d-1}}\{a,p\}(x,\xi)\,d\mu(x,\xi)=0.
	\end{eqnarray}
\end{theorem}

Finally we present the last important result we need, namely:

\begin{theorem}\label{theorem7}
	Let $P$ be a self-adjoint differential operator of order $m$ on $\Omega$ which admits a principal symbol $p$. Let $\{u_n\}_n$ be a bounded sequence in $L_{loc}^2(\Omega)$ which converges weakly to zero, with a microlocal defect measurement $\mu$. Let us assume that $P u_n$ converges to $0$ in $H_{loc}^{-(m-1)}$. Then the support of $\mu$, $\hbox{supp}(\mu)$, is a union of curves like $s\in I \mapsto \left(x(s), \frac{\xi(s)}{|\xi(s)|}\right)$, where $s\in I \mapsto (x(s),\xi(s))$ is a bicharacteristic of $p$.
\end{theorem}

Let us consider the wave operator in a inhomogeneous  medium:
\begin{eqnarray*}
	\rho(x) \partial_t^2 - \sum_{j=1}^n \partial_{x_j}\left[K(x) \partial_{x_j}\right].
\end{eqnarray*}

Using the notation $D_j=\frac{1}{i}\partial_j$ we can write
\begin{eqnarray*}
	P(t, x, D_t, D_{x})= - \rho(x) D_t^2 +   \sum_{j=1}^n D_{x_j}[K(x)D_{x_j}],\quad D_x=(D_{x_1},\cdots, D_{x_n}),
\end{eqnarray*}
whose principal symbol $p(t,x,\tau,\xi)$ is given by
\begin{eqnarray}\label{5.5}
p(t,x,\tau,\xi)=-\rho(x) \tau^2 + K(x) \,\xi\cdot \xi,\quad\xi=(\xi_1,\cdots,\xi_n),
\end{eqnarray}
where $t\in \mathbb{R}$,~ $x\in \Omega\subset \mathbb{R}^d$,~ $(\tau,\xi)\in \mathbb{R}\times \mathbb{R}^d$.

Next, we shall describe the bicharacteristics of $p$. These do not change if we multiply $p$ by a non-null function, rather study the hamiltonian curves of
\begin{eqnarray}\label{5.6}
\tilde p(t,x,\tau,\xi)= \frac12 \left(\frac{K(x)}{\rho(x)}\,\xi\cdot\xi-\tau^2 \right).
\end{eqnarray}

We have
\begin{equation}\label{5.7}
\left\{
\begin{aligned}
&\dot{t} = \frac{\partial \tilde p}{\partial \tau}=-\tau,\\
&\dot{x} = \frac{\partial \tilde p}{\partial \xi}= \frac{K(x)}{\rho(x)}\xi,\\
&\dot{\tau}=-\frac{\partial \tilde p}{\partial t}=0,\\
&\dot{\xi}=-\frac{\partial \tilde p}{\partial x}=-\frac12 \nabla\left(\frac{K(x)}{\rho(x)}\right)(\xi\cdot\xi).
\end{aligned}
\right.
\end{equation}

Introducing the function $G(x):=\left(\frac{K(x)}{\rho(x)}\right)^{-1}$, equations in $(x,\xi)$ become in
\begin{eqnarray}\label{5.8}
\dot{x}=K(x)\,\xi \Longleftrightarrow \xi=G(x) \dot{x},
\end{eqnarray}
which implies, considering the notation $\left<x,\xi\right>=x\cdot \xi$, that
\begin{eqnarray*}
\dot{\xi}=\left(G(x)\dot{x}\right)^{\cdot}
&=& \frac12 \nabla G(x)\,\dot{x}\cdot\dot{x}.\nonumber
\end{eqnarray*}

Once $\tilde p$ is null on each the bicharacteristic curves from \eqref{5.5} we deduce that
\begin{eqnarray}\label{5.9}
\frac{K(x)}{\rho(x)}\,\xi\cdot \xi = \tau^2=\hbox{ constant on the curve},
\end{eqnarray}

On the other hand by \eqref{5.8} and \eqref{5.9}, one has
\begin{eqnarray}\label{5.10}
G(x)\dot{x}\cdot\dot{x} =\frac{K(x)}{\rho(x)}\,\xi\cdot \xi.
\end{eqnarray}

Combining \eqref{5.9} and \eqref{5.10} we deduce
\begin{eqnarray}\label{5.11}
G(x)\dot{x}\cdot\dot{x}=\frac{K(x)}{\rho(x)}\,\xi\cdot \xi = \tau^2=\hbox{ constant on the curve},
\end{eqnarray}
i.e., the quantity $ G(x)\,\dot{x}\cdot\dot{x}$ is preserved under the flow.

By \eqref{5.9} and \eqref{5.11} we can write
\begin{eqnarray}\label{5.12}
\frac{d}{ds}\frac{G(x)\dot{x}}{\sqrt{G(x)\,\dot{x}\cdot \dot{x}}}=\frac12 \frac{\nabla G(x)\,\dot{x}\cdot \dot{x}}{\sqrt{G(x)\,\dot{x}\cdot \dot{x}}}.
\end{eqnarray}

Setting
\begin{eqnarray*}
	L(x,\dot{x}):=\sqrt{G(x)\,\dot{x}\cdot\dot{x}},
\end{eqnarray*}
from the last identity yields
\begin{eqnarray*}
	\frac{d}{ds}\frac{\partial}{\partial \dot{x}} L(x,\dot{x})=\frac{\partial}{\partial x}L(x,\dot{x}),
\end{eqnarray*}
which is the \underline{Euler-Lagrange} equation associated to $L$, namely, the \underline{geodesic} equation for the metric $G$ of $\Omega$.

Reversely, if
\begin{eqnarray*}
	\alpha \mapsto x(\alpha),
\end{eqnarray*}
is a geodesic for the metric $G$ on $\Omega$, and parameterizing the curve $x$ by the curvilinear abscissa $\sigma$ defined by
\begin{eqnarray}\label{5.13}
\frac{d \sigma}{d \alpha}=\sqrt{G(x(\alpha))\,\dot{x}(\alpha)\cdot \dot{x}(\alpha)},
\end{eqnarray}
the equation \eqref{5.12} becomes
\begin{eqnarray}\label{5.14}
\frac{d}{d\tau}\left(G(x) \frac{dx}{d\tau} \right)=\frac12 \nabla G(x)\, \frac{dx}{d\tau}\cdot \frac{dx}{d\tau},
\end{eqnarray}
with
\begin{eqnarray*}
	G(x)\,\frac{dx}{d\sigma} \cdot \frac{dx}{d\sigma}=1.
\end{eqnarray*}

We find so \eqref{5.7}, \eqref{5.8} and \eqref{5.9} setting, for instance, $s=-\frac{\sigma}{\tau}$. We observe that $\frac{dt}{d\sigma}=1$. As a conclusion, we have proved the following result:

\begin{proposition}\label{proposition5.1}
	Unless a change of variables, the bicharacteristics of \eqref{5.5} are curves of the form
	\begin{eqnarray*}
		t \mapsto \left(t, x(t), \tau, -\tau\left(\frac{K(x(t))}{\rho(x(t))}\right)^{-1}\dot{x}(t) \right),
	\end{eqnarray*}
	where $t \mapsto x(t)$ is a geodesic of the metric $G=\left(A\right)^{-1}=(\mathbf{g}_{ij})$ on $\Omega$, parameterized by the curvilinear abscissa.
\end{proposition}

\subsection{Part II}
In what follows we shall give examples of metrics $G=\left(\frac{K}{\rho}\right)^{-1}$ such that its bicharacteristics enter in the region $\omega$ where $a(x)>0$.
We consider a Riemannian manifold $M$ with Riemannian metric $G=\left< \cdot, \cdot\right>$ and Riemannian connection $\tilde\nabla$.
We denote its Laplace-Beltrami operator by $\tilde \Delta$.
Fix a coordinate system $(x_1,\ldots, x_n)$ on $M$, denote $G_{ij}=\left<\partial /\partial x_i,\partial /\partial x_j\right>$, let $G^{ij}=K_{ij}/\rho$ be the inverse matrix of $G_{ij}$ and set $\rho =\sqrt{\det (G_{ij})}$. The Laplace-Beltrami operator in this coordinate system is given by
\[
\tilde \Delta u=\frac{1}{\sqrt{\det G_{ij}}}\sum_{i,j=1}^n \frac{\partial}{\partial x_i}\left( \sqrt{\det G_{ij}}G^{ij} \frac{\partial u}{\partial x_j} \right)=\frac{1}{\rho(x)} \text{div}(K(x)\nabla u)
\]
where $\nabla$ is the usual gradient correspondent to the Euclidean metric on the domain $(x_1,\ldots, x_n)$.

\begin{remark}\label{remark6}
	It is straightforward that $\det K=\rho^{n-2}$.
	For $n\geq 3$, it implies that there exist a bijection between the Riemannian metrics $G$ and the operators $K$ given by $G\mapsto K=\rho G^{-1}$ and $K \mapsto G= (\det K)^{\frac{1}{n-2}}K^{-1}$.
	But for $n=2$, we always have $\det K=1$ and $K$ does not determine the Riemannian metric.
	In fact, observe that we have $K=Id$ in the hyperbolic plane $\mathbb{H}^2$ given in Example \ref{example5.1} below and in the half plane $\mathbb R^2_+$ endowed with Euclidean metric. Although they have the same $K$, they are not isometric because $\mathbb{H}^2$ have constant Gaussian curvature equal to $-1$ and the Euclidean open subset $\mathbb{R}^2_+$ have constant Gaussian curvature  equal to $0$.
\end{remark}

The Hessian of a smooth function $\phi:M\rightarrow \mathbb R$ is a symmetric 2-form on $M$ defined as
\begin{equation}\label{5.15}
\tilde\nabla^2 \phi(X,Y)=XY(\phi)-\tilde\nabla_XY (\phi),
\end{equation}
where $X$ and $Y$ are vector fields on $M$.
Here $X(\phi)$ denotes the directional derivative of $\phi$ in the direction of the vector field $X$.
It is well known that the value of $\nabla^2(X,Y)(p)$ depends only on the values of $X$ and $Y$ on $p$.
It means that the right-hand side of \eqref{5.15} does not depend on the smooth extension we take for $X(p)$ and $Y(p)$.

A curve $\gamma:(-\varepsilon,\varepsilon) \rightarrow M$ is a geodesic if $\tilde \nabla_{\gamma^\prime (t)}\gamma^\prime (t)\equiv 0$.

The next lemma states the main technical condition we need in order to let some subsets outside $\omega$ without damping.

\begin{lemma}\label{lemma5.1}
	Let $M$ be a complete Riemannian manifold, eventually with boundary, and let $\phi:\text{int}M\rightarrow \mathbb R$ be a smooth function. Suppose that $\phi$ is bounded and $\nabla^2 \phi(v,v)\geq c\Vert v \Vert^2$ on an open subset $U\subset \text{int} M$, where $c>0$ is a constant. Then any geodesic on $U$ hits $\partial  U$.
\end{lemma}
\begin{proof}
	Let $\gamma$ be a geodesic on $U$. Then
    \begin{align*}
	\frac{d^2}{dt^2}\phi(\gamma(t))= {} &\gamma^\prime(t) \gamma^\prime(t) \phi\\
	= {} &\tilde \nabla^2 \phi(\gamma^\prime (t),\gamma^\prime (t))+ (\tilde \nabla_{\gamma^\prime(t)}\gamma^\prime(t))(\phi)\\
	= {} & \tilde\nabla^2 \phi(\gamma^\prime (t),\gamma^\prime (t))\geq c \Vert \gamma^\prime(t)\Vert^2
	\end{align*}
	where the last equality holds because $\gamma$ is a geodesic. Observe that the last term is a positive real constant because $\Vert \gamma^\prime(t) \Vert$ does not depend on $t$.
	
	Then $\phi(\gamma(t))$ is a smooth real valued function which second derivative is  bounded below by a strictly positive constant.
	Then it is not difficult to prove that if $\gamma$ is defined on a interval $(a,\infty)$, then $\lim_{t\rightarrow \infty} \phi(\gamma(t))=\infty$.
	Analogously $\lim_{t\rightarrow -\infty} \phi(\gamma(t))=\infty$ whenever $\gamma$ is defined on a interval $(-\infty,a)$.
	But neither of the cases are possible if $\gamma$ remain forever in $U$ because $\phi$ is bounded there.
	Therefore $\gamma$ must hit $\partial U$.
\end{proof}

We would like to know if $W \subset \Omega-\omega$ can be left without damping.
If $\bar W$ is contained in an open subset $V\subset \text{int}\Omega$ with Riemannian metric $G$ and $\phi:V\rightarrow \mathbb R$ is a smooth function such that $\tilde \nabla^2 \phi(v,v)> 0$ for every $v \neq 0$, then $W$ can be left without damping.
In fact, we are able to find an open subset $U\subset \subset V$ containing $W$, $\phi\vert_U$ is bounded, satisfies $\tilde \nabla^2 (\phi\vert_U)(v,v)\geq c\Vert v\Vert^2$ for some $c>0$ and we can use Lema \ref{lemma5.1} for $\phi\vert_U$.

A particular a common instance is the case when $W_1\subset V_1, \ldots, W_m\subset V_m$ is a pairwise disjoint family of compact subsets in $\Omega-\omega$, where each $W_i$ satisfies the conditions of $W$ of the former paragraph and $\phi_i:V_i \rightarrow \mathbb R$ are the respective smooth functions with positive Hessian.
We can choose $V_1,\ldots , V_m$ such that they are pairwise disjoint.
Let $U_i \subset \subset V_i$ be open subsets containing $W_i$.
Due to the differentiable version of the Urysohn's lemma, there exist a differentiable function $\eta:\Omega \rightarrow \mathbb R$ such that $\eta \equiv 1$ on $\cup_i W_i$ and $\eta \equiv 0$ on $\Omega - \cup_i U_i$.
Now define
\[
\phi(x)=
\left\{
\begin{array}{c}
\eta (x)\phi_i(x) \text{ if }x\in V_i \\
0 \text{ \ \ \ otherwise.}
\end{array}
\right.
\]
Then $\phi$ is a bounded smooth function which has positive Hessian in a neighborhood of $\cup_i W_i$ and we have proved the following useful result:

\begin{proposition}\label{proposition5.2}
	Let $\{W_i\}_{i=1,\ldots, m}$ a family of pairwise disjoint compact subsets of $V-\omega$.
	Suppose that there exist neighborhoods $V_i$ of $W_i$ and smooth functions $\phi:V_i \rightarrow \mathbb R$ with positive Hessian. Then $\cup_i W_i$ can be left without damping.
\end{proposition}

Of course $W=\Omega - \omega$ is a particular case of Proposition \ref{proposition5.2}.

Let $M$ be a Riemannian manifold and $p\in M$.
Let $d:M\rightarrow \mathbb R$ be the distance function from $p$.
Then $d^2$ is a smooth function defined on the ball $B(p,r)$ inside the injectivity radius of $p$.
Moreover we can control $r$ such that $d^2$ have positive Hessian on $B(p,r)$.
More precisely we have the following well theorem:

\begin{theorem}[See \cite{JOST}] \label{theorem8}
	Let $M$ be a Riemannian manifold, consider $p\in M$ and let
	\[
	B(0,r):=\{ v\in T_pM \quad :\quad \| v\|\leq r\}
	\]
	such that the exponential map $\exp_p:B(0,r) \to M$ is a diffeomorphism over its image. Denote by $B(p,r)=\exp_p(B(0,r))$ the geodesic ball with center $p$ and radius $r$.
	Suppose that the sectional curvature $K$ on $B(p,r)$ satisfies the inequality
	\[
	K\leq \mu,\hspace{0.5cm} \mu\geq 0,
	\]
	and that
	\begin{equation}\label{5.16}
	r<\frac{\pi}{2 \sqrt{\mu}}
	\end{equation}
	if $\mu > 0$. Then the  square of the distance to $p$ function $d^2(\cdot,p):B(p,r)\rightarrow \mathbb R$ has positive definite Hessian.
\end{theorem}

We can use this theorem in order to construct Riemannian metrics on $\Omega$ such that $V-\omega$ can be left without damping.
For instance, if $(q_1,\ldots, q_d)\in \mathbb R^d$ is a fixed point, then we can consider the function $d^2(x_1,\ldots, x_d)=\sum_{i=1}^d(x_i-q_i)^2$ restricted to $\Omega$ and put a Riemannian metric with sufficiently small constant curvature $K$ in $B(q,r)\supset \Omega - \omega$ such that $B(q,r)$ satisfies the conditions of Theorem \ref{theorem8}. Then $V-\omega$ can be left without damping.

\begin{example}\label{example5.1}
	Denote the hyperbolic plane by $\mathbb{H}^2$. The half plane model of $\mathbb{H}^2$ is given by the set $\mathbb{R}^2_+ = \{(x_1,x_2) \in \mathbb{R}^2; x_2>0\}$ endowed with the Riemannian metric $G_{ij}=(1/x_2^2) \delta_{ij}$, where $\delta_{ij}$ is the Kroenecker delta function.
	The geodesics in $\mathbb{H}^2$ are the vertical lines and the semicircles with the origin in the $x$-axis.
	Every geodesic is a minimizer, that is, the length of a geodesic that connects $x$ and $y$ is the distance between $x$ and $y$.
	Moreover curves that are arbitrarily close to the $x$-axis (for instance, geodesics in $\mathbb{H}^2$) have infinite length.
	
	Bounded domains $\Omega$ in $\mathbb H^2$ have compact closures in $\mathbb H^2$.
	From all these facts, it is straightforward that every geodesic does not remain in any bounded domain.
	
	Alternatively we could observe that $\mathbb{H}^2$ has constant Gaussian curvature equal to $-1$ and that the exponential maps in $\mathbb{H}^2$ are diffeomorphisms $\exp_p:T_pM \rightarrow \mathbb H^2$ for every $p\in \mathbb H^2$.
	Now we can use Theorem \ref{theorem8} in order to reach the same conclusion.
\end{example}

\end{document}